\documentclass[12pt]{article}
\usepackage[utf8]{inputenc}
\usepackage{subcaption}
\usepackage[labelfont=sc]{caption}
\usepackage[citestyle=authoryear, giveninits=true, maxbibnames=4, url=false]{biblatex}
\usepackage[utf8]{inputenc}
\usepackage{graphicx}

\usepackage{url}

\usepackage[dvipsnames]{xcolor}

\usepackage[a4paper, margin=2cm]{geometry}

\usepackage[toc,page]{appendix}
\usepackage{authblk}

\usepackage{titlesec}
\titleformat*{\section}{\large\bfseries}
\titleformat*{\subsection}{\bfseries}
\titleformat*{\subsubsection}{\itshape}
\usepackage{bbm}

\usepackage{bm}

\definecolor{DarkMidnightBlue}{rgb}{0.0, 0.2, 0.5}
\usepackage[colorlinks=true,linkcolor=Maroon,citecolor=DarkMidnightBlue,urlcolor=black]{hyperref}

\usepackage{lineno}
\usepackage{booktabs}

\usepackage{amsmath}
\usepackage{newunicodechar}

\usepackage{amsthm,amsmath}
\usepackage{amssymb,verbatim}

\newcommand{\pder}[2][]{\frac{\partial #1}{\partial #2}}

\renewcommand{\P}{\mathrm{P}}

\newcommand{\e}{\mathrm{e}}

\newcommand{\ind}{\mathbf{1}}

\newcommand{\dd}{ \mathrm{d}}

\renewcommand{\epsilon}{\varepsilon}

\renewcommand{\tilde}{\widetilde}
\numberwithin{equation}{section}

\let\originalmiddle=\middle
\def\middle#1{\mathrel{}\originalmiddle#1\mathrel{}}

\newtheorem{theorem}{Theorem}[section]
\newtheorem{lemma}[theorem]{Lemma}
\newtheorem{proposition}[theorem]{Proposition}
\newtheorem{corollary}[theorem]{Corollary}

\theoremstyle{definition}

\newtheorem{remark}[theorem]{Remark}

\newtheorem{assumption}[theorem]{Assumption}

\usepackage{algorithm}
\usepackage{algpseudocode}

\usepackage{stmaryrd}

\newcommand{\R}{\mathbb R}

\usepackage{pgfplots}
\definecolor{lime}{HTML}{A6CE39}
\DeclareRobustCommand{\orcidicon}{%
	\begin{tikzpicture}
	\draw[lime, fill=lime] (0,0) 
	circle [radius=0.16] 
	node[white] {{\fontfamily{qag}\selectfont \tiny ID}};
	\draw[white, fill=white] (-0.0625,0.095) 
	circle [radius=0.007];
	\end{tikzpicture}
	\hspace{-2mm}
}

\foreach \x in {A, ..., Z}{%
	\expandafter\xdef\csname orcid\x\endcsname{\noexpand\href{https://orcid.org/\csname orcidauthor\x\endcsname}{\noexpand\orcidicon}}
}

\makeatletter
\newcommand*\bigldot{\mathpalette\bigldot@{.4}}
\newcommand*\bigldot@[2]{\mathbin{{\hbox{\scalebox{#2}{$\m@th#1\bullet$}}}}}

\newcommand{\bigdot}{} 
\DeclareRobustCommand\bigdot{%
  \mathbin{\mathpalette\bigdot@{0.5}}%
}
\newcommand{\bigdot@}[2]{%
  \vcenter{\hbox{\scalebox{#2}{$\m@th#1\bullet$}}}%
}
\makeatother

\newcommand{\xmiddle}[1]{\,\middle#1\,}

\input{expand.tex}

\usepackage{pdfpages}

\renewcommand{\phi}{\varphi}
\newcommand{\spectral}{|\!|}

\title{\bf Guided smoothing and control for diffusion processes}

\author{\small\bf Oskar Eklund, Annika Lang\orcidA{}\!, Moritz Schauer\thanks{Contact: \url{smoritz@chalmers.se}}\,\orcidD{}}
\affil{\small Department of Mathematical Sciences, Chalmers University of Technology and University of Gothenburg}

\date{}
\pgfplotsset{compat=1.16} 
\begin{document}
\maketitle
 
\begin{abstract} 
The smoothing distribution is the conditional distribution of the diffusion process in the
space of trajectories given noisy observations made continuously in time. It is
generally difficult to sample from this distribution. We use the theory of enlargement
of filtrations to show that the conditional process has an additional drift term
derived from the backward filtering distribution that is moving or guiding the
process towards the observations. This term is intractable, but its effect can be
equally introduced by replacing it with a heuristic, where importance weights correct for the discrepancy. From this Markov Chain Monte Carlo and
sequential Monte Carlo algorithms are derived to sample from the smoothing distribution. 
The choice of the guiding heuristic is discussed from an optimal control perspective and evaluated. The results are tested numerically on a stochastic differential equation for reaction-diffusion.

\end{abstract}

\section{Introduction}
Diffusion processes are continuous-time probabilistic models with applications across various fields, including finance, physics, and engineering, see, e.g., \cite{Oksendal2003,MT04, sarkka2019applied}. They naturally emerge by introducing Gaussian random perturbations (white noise) to deterministic systems. We consider the problem of smoothing, that is, determining the conditional distribution of the process' trajectory given observational data by a variant of Bayes' rule.  
Intuitively, this conditional distribution represents the knowledge about the stochastic trajectory given the observation data. 
We start with introducing our most central stochastic processes.
Let $X=(X_t, t \in {[0, T]})$  be a $d$-dimensional diffusion process of interest, which is indirectly observed as detailed below. Its dynamics are governed by the stochastic differential equation (SDE)
\begin{align}
\label{SDE}
    \dd X_t = b(t,X_t) \, \dd t + \sigma(t,X_t) \, \dd W_t, 
\end{align}
on the time interval $t \in [0,T]$ with $T>0$ some final time, initial value $X_0 = x_0$, and where $W$ is a $d'$-dimensional Brownian motion. Here, the drift coefficient $b\colon [0,T]\times \R^d \to \R^d$ and the dispersion coefficient $\sigma\colon [0,T] \times \R^d \to \R^{d \times d'}$ are sufficiently smooth functions.  We denote $a = \sigma \sigma'$.

Let $Y=(Y_t, t \in {[0, T]}) $ be the $D$-dimensional observation process with initial value $Y_0 = 0$ and dynamics governed by the SDE
\begin{align}
\label{obs}
    \dd Y_t = H_t X_t \, \dd t + \dd \beta_t.
\end{align}
Here, for each $t \in [0,T]$, $H_t \in \R^{D\times d}$ is a linear observation operator and $\beta$ is a $D$-dimensional Brownian motion on our probability space, independent of $W$ and $X$.


Assume that $Y$ was observed (up to time~$T$) and the conditional distribution (or smoothing distribution) of $X$ given $Y$ is of interest. In particular, we are interested in sampling paths of $X$ given $Y$ and computing conditional expectations with respect to the smoothing distribution,
\[
\EE\left[f(X) \mid Y\right],
\]
correctly up to Monte Carlo error. For suitable functions~$f$ this include for example quantities of interests such as the mean and variance of the unobserved process $X$ given the data.

The problem is intractable except in a few easy cases.
For example, when $b$ is linear in space and $\sigma$ is constant in space, we can use the Rauch--Tung--Striebel smoother, see \cite{rauch1965maximum}, to compute it.
Naively, smoothing an approximation of the process gives an approximation to the smoothing distribution.
But such direct approximations themselves are not suitable for Monte Carlo methods, as in general the conditional law and the approximation are (as measures) mutually singular, precluding the use of the approximation as proposal distribution in the Monte Carlo context when samples of the smoothing distribution are sought.   

As solution for this, we propose to use ``guided processes'' (\cite{ABFFG}) for sampling from the actual smoothing distribution. They are also derived from tractable approximations to the smoothing distribution,
for example via the Rauch--Tung--Striebel smoother of a linearized auxiliary system, albeit in a more indirect manner: 
The auxiliary process is only used to change the drift component of the process $X$ by a linear term, preventing the measures to become mutually singular. 

The process with changed drift approximates the conditional process $X$ given observations well and its samples can be used to sample the \emph{actual} conditional process by Monte Carlo, up to Monte Carlo error.

High frequency observations are for example relevant in automated stock trading, but also observations at 50-year intervals such as glacier data can constitute high frequency data as geological processes happen on time scales of million years.
Due to their large dimension, observation vectors in these applications are better modelled as functional observation data, and likewise the latent process as continuous time processes, see, e.g., \cite{sarkka2013}.

\subsection*{Method and results}

Here, we describe our method of computing expectations with respect to the smoothing distribution. We work on the complete filtered probability space $(\Omega, \cA, (\cG_t), \PP)$ supporting the independent Brownian motions $W$ and $\beta$ and consider the processes $X$ and $Y$ defined on the fixed time interval~$[0,T]$.

In section~\ref{sec:theory_guided}, using the theory of enlargement of filtrations, we introduce a filtration $\cH = (\cH_t)$ containing information about~$Y$. Under the new filtration, the meaning of ``adapted Brownian motion'' changes, which is reflected by a change in the drift of the process $X$.
In our first main theorem, we show that $X$ under $\PP$ solves the stochastic differential equation 
\begin{equation}\label{conditional}
    \dd X_t = b(t, X_t) \, \dd t + \sigma(t, X_t) (U_t \, \dd t +   \dd W^\star_t), \tag{$\star$}
\end{equation}
for a data dependent, $\cH$-adapted control process $U = (U_t, t \in [0,T])$  taking values in $\RR^{d'}$,
where $W^\star$ is again a $d'$-dimensional Brownian motion, but now also adapted to $\cH$ under $\PP$.
The theorem includes a characterisation of the control process $U$ which induces the dependence on $Y$, see \eqref{UV}.

Note that $X$ in \eqref{conditional} is still the same process under the same measure $\PP$, just expressed differently to make it possible to sample $X$ conditional having seen $Y$ already.
For this, note that the new Brownian motion $W^\star$ can be simulated in the filtration $\cH$ like any other Brownian motion and numerical solutions to \eqref{conditional} are obtained using standard integration procedures. 

The new representation makes it therefore straight forward to sample $X$ given $Y$ 
if one can get hold on the control process $U$ by solving \eqref{UV}.
However, this is not always practical, and we may settle for a suboptimal, but tractable data dependent control term $U^\circ$ to substitute~$U$. 
In section~\ref{sec:guided_next}, we define the guided process, which solves the SDE
\begin{equation}
    \dd X_t = b(t, X_t) \, \dd t + \sigma(t, X_t) (U^\circ_t \, \dd t  + \dd W^\circ_t), \tag{$\circ$}\label{guided}
\end{equation}
where $W^\circ$ is again a $d'$-dimensional Brownian motion adapted to $\cH = (\cH_t)$ under a different probability measure we denote $\PP^\circ$. 

We choose $U^\circ$ such that $X$ given $Y$ can be easily sampled under $\PP^\circ$. 
For that, we approximate the SDE~\eqref{SDE} by a linear, and therefore tractable, process, for which the corresponding control process is known in closed form.
This generalizes the backward filtering forward guiding paradigm of \cite{Mider2021} to continuous time observation models. 

Under $\PP^\circ$, $X$ solves the smoothing problem only approximately. But we then derive the Girsanov change of measure which changes the control back to $U$, expressed as function which gives each sample $X$ of \eqref{guided} a weight $\Psi(X)$.
This is our second main theorem.

In line with the theory of sequential Monte Carlo, the change of measure is tractable up to an unknown constant. 
This allows us to use Monte Carlo methods to sample the actual conditional process.
We devise a Markov Chain Monte Carlo sampler using solutions to $\eqref{guided}$ as proposals which produces dependent samples of $X \mid Y$ when it has reached stationary.

Ideally, we choose $U^\circ$ such that $\PP^\circ$ is close to~$\PP$. 
To do so in a principled manner, we parametrise $U^\circ_\theta$ and $\PP^\circ_\theta$ and vary $\theta$ as to minimise the Kullback--Leibler divergence between $\PP$ and $\PP^\circ_\theta$.  This variational problem can be formulated as optimal control problem similar to those considered in \cite{KimMehta}.

\subsection*{Notation}
We use $\P_X$ with a subscript to denote the law of a random process $X$ (i.e.~the push forward of $\PP$ under $X$). We denote the conditional distribution of $X$ given $Y$ as $\P_{X \mid Y}$. 
Expectations under a decorated measure, e.g.~under $\PP^\circ$, are indicated by a decorated expectation $\EE^\circ$, likewise the law of $X$ is denoted $\P^\circ_X$, etc.  We also denote the integral  $\int_0^\cdot Y_t \, \dd X_t$ by  $Y \bigdot X$. The processes $[X,Y]$ and $[X]$ denote the quadratic (co-)variation process with coordinate processes $[X,Y]^{(i,j)} = [X^{(i)}, X^{(j)}]$ and
$\cE(X)_\cdot = \exp\left(X_\cdot - \frac12[X]_\cdot\right)$ the Dol\'eans--Dade exponential process, up to a scaling factor the unique solution $Z$ of $\dd Z_t = Z_t \, \dd X_t$, $Z_0 = 1$. By $\sigma(Y)$ respective $\sigma(\cC)$, we denote the $\sigma$-field generated by some random variable~$Y$ or by the sets in a collection $\cC$, and we let $\cF \vee \cG := \sigma(\cF \bigcup \cG)$.
 
\subsection*{Related work}
The problem of smoothing and filtering for the model \eqref{SDE} and \eqref{obs} is well-studied. The exact smoothing distribution as well as the filtering distribution are in general unknown, although it can be found in the Gaussian affine case, see \cite{Kalman}, \cite{KalmanBucy}, \cite{rauch1965maximum}. 
In general, indirect observation of diffusion processes causes changes in the underlying probability measure which are described by Girsanov's theorem. These changes, as done in this work, can be alternatively, but to large extend equivalently, seen as changes in the information flow that can be modelled via enlargement of filtrations, see \cite{JeulinYor}.

Alternatively to both approaches, the Girsanov framework or the framework of enlargement of filtrations, conditional processes can be obtained as solutions to optimal control problems under Kullback--Leibler loss. In this direction, the conditional process can be derived as optimal controlled process, see \cite{KimMehta}, 
with
\begin{equation*}
U_t = \sigma'(t, X_t)\nabla \log v(t, X_t),
\end{equation*}
being the optimal control defined in terms of the solution $v$ of a backward stochastic partial differential equation derived in \cite{Pardoux}. The use of stochastic partial differential equations brings additional complexity which we can avoid in the main line of our work. Nevertheless, we will explore the connection in more detail.

Earlier, similar approaches based on guided processes have been used in the case of discrete-time observation processes, see \cite{papaspiliopoulos} and \cite{Mider2021}.These approaches are based on work on bridge processes derived in \cite{DelyonHu2006} and \cite{10.3150/16-BEJ833}, which generalise the Brownian bridge. Here, we extend the theory to continuous time observations. Also in this line of work, one can choose between the perspectives of changes of measures and of enlargement of filtrations, see \cite[p. 125]{Schauer2015-as} and \cite{marchand:tel-00733301}.

As reference literature for the smoothing and filtering problems, we use \cite{Bain2009}, \cite{Oksendal2003}, \cite{sarkka2013} and \cite{Handel}. For filtering techniques using machine learning or deep learning methods, see \cite{Han2018}, \cite{Bagmark}. For related work on optimal control, see e.g.~\cite{TzenRaginsky}.

\subsection*{Outline}
This paper is organized as follows: In section~\ref{sec:theory_guided}, we derive a stochastic differential equation for the smoothed process and relate it to a backward stochastic partial differential equation for which we give a solution in the case of linear processes.
In section~\ref{sec:guided_next}, we define the guided proposal process, derive the change of measure between the conditional process and the proposal process, and illustrate how to use this in sampling and in variational inference. The last section~\ref{sec:simulations} is devoted to a numerical experiment. 
The proofs of theorems are included in the main text, remaining proofs are gathered in appendix~\ref{app:proofs}.

\section{Smoothing for continuously observed diffusions}\label{sec:theory_guided}

In this section, we derive a stochastic differential equation for the smoothed process using the approach of \emph{initial enlargement of filtration} (grossissement de filtration initiale), see \cite{JeulinYor}.
Let $\cF = (\cF_t, t \in [0,T])$ (recall that $\cG$ contains information about both Brownian motions $W$ and $\beta$) be the complete and right continuous filtration created by~$W$.
The enlarged filtration $\cH = (\cH_s)$, also complete and right continuous,
\[
\cH_s = \bigcap_{t > s} (\cF_t \vee \sigma(Y)),
\]
models the flow of information when the entire realization~$Y$ is made known from the beginning of time. Now the process seen as semimartingale under the new filtration will have a new canonical decomposition which we derive in this section. 

For the simpler case where $Y = X_T$, the process~$X$  
under the filtration $\cH$ (with $\sigma(Y)= \sigma(X_T)$ accordingly), 
fulfills equation~\eqref{conditional} with
\begin{equation}\label{discrete-obs}
U_t = \nabla_x \log v(t, X_t; T, X_T),
\end{equation}
where $v(t, x; T, y) =  \PP(X_T \in {\dd y} \mid X_t = x)\big/\dd y$ is the transition density, that is the likelihood to still observe $Y = X_T$ when the process is in $x$ at time~$t$.

In the case of continuous observations considered here we can express $U$ similarly in terms of a density of the conditional distribution of~$Y$ of an event~$A$,
\begin{equation*}
\P_{Y \mid \cF_t}(A) = 
\EE [\ind_A(Y) \mid \cF_t],
\end{equation*}
which is derived in the next section using a decoupling change of measure.

\subsection{Decoupling change of measure for the joint process}

We introduce the progressive log-likelihood process $\varphi = (\varphi_t, t\in [0,T])$ taking values in $\RR$ with dynamics governed by the SDE
\begin{align*}
    \dd \varphi_t = (H_t X_t)' \, \dd Y_t - \frac12 \|H_t X_t\|^2  \, \dd t
\end{align*}
and initial value $\phi_0 = 0$.  Exponentiated it gives the likelihood of seeing a sample $(Y_s, s \in [0,t])$ given $(X_s, s \in [0,t])$ so far relative to the Wiener law, so $\phi = \log \cE( (H X)' \bigdot Y)$. 

\begin{assumption}\label{weak-ass}
Assume that the drift vector~$b$ and the dispersion matrix~$\sigma$ are measurably defined on $[0, T]\times \RR^d$, with $\sigma$ bounded and that the observation matrix $H$ is bounded on $[0,T]$ and measurable. Assume that the law of the solution to \eqref{SDE} is unique. \end{assumption}

 \begin{assumption}[Novikov]\label{novikov-ass}
    There is $\epsilon > 0$ such that $\EE \exp(\epsilon \sup_{t \in [0,T]} \| X_t\|^2 ) < \infty$.
 \end{assumption}
Under this local Novikov-type assumption, compare \cite[Exercise 1.40, p.~338]{Revuz1998-hw}, using $\varphi$ as Radon--Nikodym density, we can define a new, equivalent measure
\begin{equation*}
\dd \tilde{\PP} = \e^{-\varphi_T} \, \dd \PP,
\end{equation*}
and the process $\cE( (H X)' \bigdot Y)$ is an $\cF$-adapted martingale under $\tilde \PP$.

Under this new measure, the process~$X$ will be decoupled from the observation process~$Y$ as this is the change of measure which removes the drift term $H_tX_t \, \dd t$ from~$Y$. But the marginal law $\P_X$ of $X$ remains the same, for example $\EE [f(X)] = \tilde \EE [f(X)]$ for path functionals~$f$.

Then, under the new measure, by Girsanov's theorem,
$(W_t, \tilde \beta_t) := (W_t, \beta_t + \int_0^t H_s X_s \, \dd s) = (W_t, Y_t)$ is an $\RR^{d'+ D}$-dimensional Brownian motion and $X$ is independent of $Y$. Hence,
\[
 \tilde \P_{{X},Y} =  \P_{{X}} \otimes \tilde \P_{ Y}, 
\]
where $\tilde \P_{Y}$ is the Wiener measure restricted to $[0, T]$ and a tilde indicates that we consider the push forward under $\tilde \PP$. 
As in \cite[p.~150, (1.6)]{Pardoux}, we have the following characterization of the conditional expectation given $Y$ in form of a Kallianpur--Striebel type formula. 
\begin{proposition}
For bounded measurable real test functions $f$ on the path space $C([0,T], \R^d)$, it holds
\begin{equation*}
\EE \left[f(X) \mid Y\right] = \tilde \EE\left [f(X) Z_T
\mid Y\right],
\quad \text{ where } Z_t = \frac{ \e^{\varphi_t}}{\tilde \EE[ \e^{\varphi_t} \mid Y]}. 
\end{equation*}
\end{proposition} 

\subsection{The dynamics of the conditional process}

The process $V = (V_t, t \in [0,T])$, given by
\[V_t = \tilde \EE\left[ \exp(\phi_T - \phi_t) \mid \cH_t \right],\]
can be understood as likelihood of the future part $(Y_r, t \le r \le T)$ of the observation $Y$ given $X_t$ and plays a similar role as the transition density $q(t, X_t, T, y)$ in the case of observing $X_T = y$.  
The dynamics of the process $X$ under the new filtration are governed by the \emph{score process}, an $\cH$-adapted process $U = (U_t, t \in [0,T]) $ such that
\begin{equation}\label{UV}
\int_0^t U_s V_s \, \dd s = [W, V]_t.
\end{equation}

\begin{theorem}\label{thm:yor}
Under assumptions~\ref{weak-ass} and~\ref{novikov-ass}, there exists a predictable integrable process~$U$ such that \eqref{UV} holds and 
\[
W^\star_t = W_t - \int_0^t U_s \, \dd s
\]
is an $\cH$-Brownian motion under $\PP$. The process $X$ from \eqref{SDE} solves~\eqref{conditional} with $U$ and $W^\star$.
\end{theorem}
\begin{proof}
Define $p_t = \Phi_t V_t = \tilde\EE[\Phi_T \mid \cH_t]$ with $\Phi_t =  \exp(\phi_t)$. The process~$p$ is an $\cH$-martingale under~$\tilde \PP$.  
From the martingale representation theorem, which holds under initial enlargement of $\cF$ by the $\tilde \PP$-independent $\sigma(Y)$, see \cite[theorem~III.4.33]{Jacod2002-ur}, 
there is a square integrable stochastic process~$\Delta$ with
\[
p_t = p_0 + \int_0^t \Delta_s \, \dd W_s.
\]
Also $[W, p]$ is absolutely continuous.
Therefore, as $p > 0$,
\[
U_t = p^{-1}_t \Delta_t  = p_{t}^{-1}\pder t [W, p]_t
\] 
solves $p = p_0 \cE(U'\bigdot W)$
by $p^{-1} \bigdot p = (p^{-1} \Delta) \bigdot W = U' \bigdot W$
and
\begin{equation}\label{Up}
\int_0^t U_s p_s \, \dd s = [W, p]_t.
\end{equation}
Integration by parts yields
\[
W p = p \bigdot W +  W \bigdot p + [W, p],
\]
and with $A_t =  \int_0^t U_s \, \dd s$ by \eqref{Up}
\[
A p =  A \bigdot  p  + p \bigdot A =  A \bigdot  p + [W, p] .
\]
Also $W$ is an $\cH$-martingale under $\tilde \PP$. So, 
\[
W^\star p = (W - A)p = (Wp - [W, p]) - (Ap - [W, p])
\]
is an $\cH$-martingale under $\tilde \PP$ as difference of two $\cH$-martingales. 
Then by (abstract) Bayes' formula
\[
\EE\left[W^\star_t   \xmiddle| \cH_s\right] = \tilde \EE\left[W^\star_t  p_t   \xmiddle| \cH_s\right]/p_s = W^\star_s. 
\]
As $[W^\star]_t = t$, by L\'evy's characterization, $W^\star$ is an $\cH$-Brownian motion under~$\PP$.  
The SDE~\eqref{conditional} follows by substituting $\dd W^\star_t + U_t\,\dd t$ for $\dd W_t$ in \eqref{SDE}.

Finally, by $[\Phi, W] = 0$ 
we have from integration by parts $[W, p]  = [W, \Phi V] 
= [W, \Phi\bigdot V + V \bigdot \Phi + [V, \Phi]]
= \Phi\bigdot  [W,  V ].$
We see that $U$ 
also solves \eqref{UV}  by $\int_0^\cdot  U_s V_s \, \dd s =  \Phi^{-1}\bigdot [W, p]  =  (\Phi^{-1}\Phi)\bigdot [W,  V ] =  [W,  V ].$
\end{proof}

The term~$U$ guides the process closer to the observations~$Y$ and plays a similar role as the corresponding term in case of a single discrete observation in~\eqref{discrete-obs}. We derive a corollary  
that further helps to gain intuitive understanding of $U$ and will help to compute $U$ in the linear Gaussian case later.
\begin{corollary}\label{parametrized}
Let $V$ be parametrised by an $r$-dimensional It\^o-process  $\Theta = (\Theta_t, t \in [0,T])$, independent of $X$,  in the sense that there is a function 
$v\colon [0,T]\times \RR^d \times \RR^r \to [0,\infty)$, continuously differentiable in space,  
such that 
\[
v(t, X_t, \Theta_t) = V_t.
\]
Then
\[
U_t =\sigma'(t, X_t) \nabla_x \log v(t, X_t, \Theta_t).
\]
\end{corollary}
So $(\sigma')^{-1} U$ is the direction where the likelihood $V$ increases,  
$\sigma U$ acts as a drift in the corresponding direction in coordinates of $X$, see \cite{corstanje2024}.

\subsection{Connection to stochastic PDEs}
Without the assumption that $V$ can be parametrised by a finite dimensional process, 
$U$ is correspondingly determined by the following backward stochastic partial differential equation (PDE), due to \cite{Pardoux, Pardoux1982, HEUNIS1990185}, in particular \cite[Sec.~5]{Pardoux1982} for the case of unbounded coefficients.
Define the formal differential operator
\begin{align*}
    \fL f(t,x) 
        & = 
        \langle b(t,x), \nabla f(t,x)\rangle + \frac12 \operatorname{Tr}(a(t,x) \operatorname{Hess}_f(t,x))\\
        & = 
        \sum_{i=1}^d b_i(t,x) \pder{x_i}  f(t,x)  + \frac12 \sum_{i, j =1}^d a_{ij}(t,x) \frac{\partial^2}{\partial x_i \partial x_j}  f(t,x)
\end{align*}
associated with $X$. 

One possible way to make practical use of theorem \ref{thm:yor} is using these results to compute $U$.
\begin{proposition}
Under the additional conditions
that $b$ is once and $\sigma$ and $\sigma \sigma'$ are twice bounded differentiable in space, $\sigma \sigma'$ is uniformly elliptic, and $H$ is bounded differentiable in time,
 \[
v(t, x) = \tilde\EE\left[\frac{\Phi_T}{\Phi_t}\middle| X_t = x, Y\right]
\]
 satisfies the stochastic PDE 
\begin{equation}\label{parabolic2}
\dd v(t, \cdot)    + \fL v(t, \cdot) \, \dd t + v(t, \cdot) (H_t  \cdot) '\diamond  \dd Y_t  = 0
\end{equation}
with terminal condition $v(T,\cdot)\equiv 1$,
where $\diamond$ indicates a backward It\^o integral.
\end{proposition} 
We refer to \cite{HEUNIS1990185} for further motivation, for the notion of a backward It\^o integral and for what entails to be a solution in this context. 
The stochastic term reads in integral form
\[\int_0^t  G(v(s, \cdot), (H_s \cdot))' \diamond \dd Y_s, \]
where for $c \colon \RR^d\to \RR$, $g \colon \RR^d\to \RR^D$, we have
$G(c, g)(x) = c(x) g(x)$.
\begin{remark}\label{terminal} When additional to $Y$ also a random variable $\zeta$ that depends (only) on $X_T$ is noisily observed, then the terminal condition changes to $v(T,\cdot) = L(x; \zeta)$, where $L(x; y)$ is the conditional probability density of $\zeta$ given $X_t = x$ (the likelihood of $x$). Then $V_T = L(X_T; \zeta)$ and
\[
v(t, x) = \tilde\EE\left[\frac{\Phi_T}{\Phi_t}\,L(X_T; \zeta)\middle| X_t = x, Y, \zeta\right].
\]
Thus one can easily extend our results to the setting where both continuous and discrete observations are available, extending \cite{Mider2021} and \cite{ABFFG} by this setting.
\end{remark}

\begin{remark}\label{palmowski}
From \eqref{parabolic2} by It\^o's formula (\cite[p.~133]{Pardoux}),
\begin{equation*}
\dd \log v(t, \cdot) +  \left[ \fL \log v(t, \cdot) + \frac12 \|\sigma' \nabla \log v(t, \cdot)\|^2\right] \, \dd t + (H_t \cdot)' \, \dd Y_t - \frac12\|(H_t \cdot)\|^2 \, \dd t    = 0.
\end{equation*}
Details on the derivation can be found in appendix~\ref{app:proofs}.
\end{remark}

\subsection{Backward filter for the likelihood process in the linear case}
Consider a linear process \eqref{SDE} with $b(t,x) = B_tx + m_t$ and $\sigma(t,x) = \sigma_t$  constant in space, where $\sigma_t \in \R^{d \times d'}$, $B_t \in \R^{d \times d}$, and $ m_t \in \R^d$. 
In line with \cite{FraserPotter1969}, the solution of the backward stochastic PDE \eqref{parabolic2} can be alternatively obtained by a backward filter. 
Introduce
\[
\fL^\dagger f(t,x) = -\sum_{i=1}^d b_i(t,x) \pder{x_i}  f(t,x)  + \frac12 \sum_{i, j =1} a_{ij}(t,x) \frac{\partial^2}{\partial x_i \partial x_j}  f(t,x)
\]
and 
let
\[
c(t)  =  -\operatorname{Tr}(B_t).
\]

The operator $(-\pder t + \fL^\dagger) f + cf$ is the formal adjoint of $\pder t + \fL$.
Now one can think of $-\pder t + \fL^\dagger$, where we split off the bounded multiplicative operator $c$, 
as the space-time generator of a linear process $X^\dagger_{t}$ running backwards in time starting from a Gaussian density $\pi^\dagger(T, \cdot)$. 

Using classical filtering theory for the process $X^\dagger_{T-r}$, we can derive the Kushner--Stratonovich equation for the proper filtering density
\[
\pi^\dagger(s, x) = \PP(X^\dagger_s \in \dd x \mid Y_t, t \in [s, T])/\dd x,
\]
where $Y_T - Y_t = \int_t^T H_s X^\dagger_s \, \dd s + \beta_T - \beta_t$.
Relating $\pi^\dagger(s, x)$ and $v(t, x)$, we obtain a guiding term from the filtering density of the $X^\dagger$ process. More specifically, there is a process $C = (C_t, t\in [0,T])$ such that $v(t, x) = \tilde\EE[\Phi_T/\Phi_t\,L(X_T)\mid X_t = x, Y, \zeta]$ in the setting of remark~\ref{terminal} with he conditional distribution of $\zeta$ given $X_T$ given by $\mathrm{N}(B_\zeta X_T, \Sigma_\zeta)$, a normal distribution with parameters $B_\zeta X_T$, $\Sigma_\zeta$, and
\[
v(s, x) = C_t \pi^\dagger(s, x),
\]
with terminal condition
\begin{equation}\label{terminal-filter}
v(T, x) = C_T \pi^\dagger(T, x),
\end{equation}
where $ v(T, x)\equiv 1$ is obtained as limiting case.

In a statistical context, $C$ can depend on unknown parameters and therefore cannot be neglected. For example, maximum likelihood estimation of $B$ and $m$ can be implemented by gradient descent on $v(0, x_0)$, in which case knowledge of $C$ is needed.
We have the following proposition proven in appendix~\ref{app:proofs} and giving directly the backward filtering equation for the likelihood process $v$ making $C$ explicit.
\begin{proposition}\label{backwardkalmanbucy}
The process~$v$ is given by
\begin{equation}
    v(t,x)
        = \frac{C_t}{\sqrt{(2\pi)^d |P_t|}} \exp\left(-\frac12(x-\nu_t)' P_t^{-1} (x-\nu_t)\right),\label{linearansatz}
\end{equation}
where the parameters $\nu_t$, $P_t$, $C_t$ are obtained from solving a system of backward stochastic differential equations
\begin{align}
\label{backwardODEslinear}
\begin{split}
\dd \nu_t &= (B_t\nu_t + m_t) \, \dd t  - P_t H_t' (\dd Y_t - H_t \nu_t \, \dd t), \\
\dd P_t &= \left(B_t P_t + P_tB_t' - a_t + P_t H_t'  H_t P_t\right) \, \dd t,\\
\dd \log C_t &= \operatorname{Tr}(B_t) \, \dd t -  (H_t \nu_t)' \diamond \dd Y_t + \frac12 \|H_t \nu_t\|^2 \, \dd s
\end{split}
\end{align}
with $a =\sigma \sigma'$ and with terminal condition \eqref{terminal-filter}, assuming that $P_t$ has a well-defined inverse for $t \in [0,T)$ with bounded limit for $t \to T$.
Consequently, the smoothing process $X$ solves
\begin{align*}
    \dd X_t= (B_tX_t + m_t) \, \dd t + a_t P_t^{-1}(\nu_t - X_t) \, \dd t + \sigma_t \, \dd W^\star_t.
\end{align*}
\end{proposition}
Note that result contains Theorem 2.6 in \cite{Mider2021} as special case for $H_t \equiv 0$.

\section{Guided processes}
\label{sec:guided_next}

The process~$U$ in theorem~\ref{thm:yor} is hardly ever tractable outside the linear Gaussian case. Hence, direct forward simulation of the conditional process by solving \eqref{conditional},
\begin{align*}
    \dd X_t &= b(t, X_t) \, \dd t + \sigma(t, X_t) U_t \, \dd t + \sigma(t, X_t) \, \dd W^\star_t \tag{\ref*{conditional}}
\intertext{
is in general not possible. 
In this section, we propose an approximation based on proposition~\ref{backwardkalmanbucy}.
For a tractable, $\cH$-adapted approximation $U^\circ$, the process $X$ solves equation~\eqref{guided},}
    \dd X_t &= b(t, X_t) \, \dd t + \sigma(t, X_t) U^\circ_t \, \dd t + \sigma(t, X_t) \, \dd W^\circ_t, \tag{\ref*{guided}}
\end{align*}
with $\dd W^\circ_t = \dd W^\star_t + (U_t  - U^\circ_t) \, \dd t =  \dd W_t - U^\circ_t \, \dd t$. 

Suppose (i) $\EE  \left[{\mathcal {E}}(( U^\circ -U)'\bigdot W^\star )_T \right] = 1$. 
Then by Girsanov's theorem, $\dd \PP^\circ = {\mathcal {E}}(( U^\circ -U)'\bigdot W^\star )_T\,\dd \PP$ defines a law under which 
$W^\circ$ is a Brownian motion, so weak solutions of \eqref{guided} can be obtained numerically under $\PP^\circ$. 
If also (ii) $\EE^\circ  \left[{\mathcal {E}}(( U -U^\circ)'\bigdot W^\circ )_T \right] = 1$,
(weighted) Monte Carlo samples from $X \mid Y$ can be obtained from those solutions by weighting them with Girsanov's likelihood $Z ={\mathcal {E}}(( U -U^\circ)\bigdot W^\circ ) _T = {\mathcal {E}}(( U^\circ -U)\bigdot W^\star )^{-1}_T$. We come back to this after the next section.

\subsection{Guiding based on a linear approximation}
It remains to choose $U^\circ$ that approximates $U$, express $Z$ containing the intractable~$U$ only as proportionality constant, and show equivalence between $\PP$ and $\PP^\circ$ assumed in $(i)$ and $(ii)$.

Let us split $b(t,x) = B_tx + m_t + g(t,x)$ into an (affine) linear part and a non-linear $g$, so that ideally $\tilde b(t, x) = B_t x + m_t$ can be thought of as approximation of $b$, likewise we approximate the dispersion coefficient $\sigma(t, x)$ by a dispersion coefficient~$\tilde\sigma_t$ constant in space. Again, let $\tilde a_t = \tilde \sigma_t\tilde\sigma'_t$.
Finally, if $\zeta$ is observed, we approximate the conditional distribution of $\zeta$ by $\mathrm{N}(B_\zeta X_T, \Sigma_\zeta)$ for some matrix $B_\zeta$ and noise covariance $\Sigma_\zeta$. If $\zeta$ was not observed, we chose $B_\zeta \equiv 0$ and  $\Sigma_\zeta$ large, compare with \cite{Mider2021}.

We leverage that we can solve the smoothing problem for the process $\dd \tilde X_t = \tilde b(t, \tilde X_t) \, \dd t + \tilde \sigma_t \, \dd W_t$
with the same observation model $\dd \tilde Y_t = H_t \tilde X_t \, \dd t + \dd \beta_t$.

Under regularity assumptions, proposition~\ref{backwardkalmanbucy} applies to the system $(\tilde X, \tilde Y)$ and the corresponding solution $(C, \nu, P)$ defines, with $Y$ substituted for $\tilde Y$, an backward filter 
\[{ \tilde v(t,x) =  \left[\frac{C_t}{\sqrt{(2\pi)^{d}|P_t|}}\exp \left(-\frac12 (x - \nu_t)' P_t^{-1} (x - \nu_t)\right)\right]}\] with $\tilde V_t = \tilde v(t, X_t)$
and
\begin{equation}\label{Ucirc}
U^\circ_t = u^\circ(t, X_t)
\end{equation}
where
\begin{align*}
u^\circ(t, x) &:= \sigma'(t, x) \nabla_x \log \tilde v(t,x)\\
&=
 \sigma'(t, x) P_t^{-1} (\nu_t - x).
\end{align*}
In fact, it is enough if the conclusion of proposition~\ref{backwardkalmanbucy} holds true.

\begin{theorem}\label{girsanov}
Under assumptions \ref{weak-ass} and \ref{novikov-ass} and assuming that $\nu$ and $P$ solve \eqref{backwardODEslinear} from 
proposition~\ref{backwardkalmanbucy} for $B, m$ and $a$ with $P_T =\Sigma_\zeta$ and $\nu_T= B_\zeta \zeta$, the measures 
$\PP$ and $\PP^\circ$ are equivalent with Radon--Nikodym derivative 
\begin{align}\label{likelihood}
\frac{\dd \PP}{\dd \PP^\circ} &=  \frac{\tilde V_0\, V_T}{V_0\,\tilde V_T } \exp\left(\int_0^T \psi(t, X^\circ_t, \nu_t) \, \dd t\right),
\end{align}
where 
\begin{align*}
\psi(t, x, \nu)&=  (b(t, x)    - \tilde b(t, x)    )' P_t^{-1} (\nu-x)  \\
&\quad- \frac12\operatorname{Tr}((a(t, x) - \tilde a_t) P_t^{-1})  + \frac12 ( \nu - x)'  P^{-1}_t (a(t, x) - \tilde a_t) P^{-1}_t (\nu-x)    .   
\end{align*}

For bounded, measurable, real-valued test functions $f$ defined on the path space $ C([0,T], \RR^d)$,
\begin{align*}
\EE\left[ f(X) \mid Y \right] &= \frac{\tilde V_0\, V_T}{V_0\,\tilde V_T } \EE\left[f(X^\circ)  \exp\left(\int_0^T \psi(t, X^\circ_t, \nu_t) \, \dd t\right) \xmiddle| Y \right]. 
\end{align*}
When $\zeta$ was observed, we have
\begin{align*}
\EE\left[ f(X) \mid Y \right] &= \frac{\tilde V_0\, V_T}{V_0\,\tilde V_T } \EE\left[f(X^\circ)  \exp\left(\int_0^T \psi(t, X^\circ_t, \nu_t) \, \dd t\right) \xmiddle| Y, \zeta\right]. 
\end{align*}

\end{theorem}
\begin{proof}
Note that $\PP$ and $\PP^\circ$ coincide on~$\cH_0$.

By Yor's formula ${\mathcal {E}}(X){\mathcal {E}}(Y-[X,Y])={\mathcal {E}}(X+Y)$,
\begin{align*}
 {\mathcal {E}}((-U + U^\circ )'\bigdot W^\star )_t &=   {\mathcal {E}}(-U' \bigdot W^\star)_t{\mathcal {E}}\left((U^\circ)' \bigdot W^\star + [U'  \bigdot W^\star,(U^\circ)' \bigdot W^\star]\right)_t 
\\
&= \frac{p_0}{p_t}{\mathcal {E}}(- (U^\circ)' \bigdot W^\circ )^{-1}_t,
\end{align*}
where we used $ p = p_0\cE(U' \bigdot W) = p_0\cE(-U' \bigdot W^\star)^{-1}$ by \eqref{Up} 
and {\small
\begin{align*}
& \cE\left((U^\circ)' \bigdot W^\star + [U'  \bigdot W^\star,(U^\circ)' \bigdot W^\star]\right)\\
& \qquad = \exp\left( 
(U^\circ)' \bigdot W^\circ + \int_0^\cdot 
 ( U^\circ_s - U_s)'U_s^\circ \, \dd s + [U'  \bigdot W^\star,(U^\circ)' \bigdot W^\star] - \frac12 [(U^\circ)' \bigdot W^\star]
\right)\\
& \qquad = \exp\left( 
(U^\circ)' \bigdot W^\circ +  \frac12 [(U^\circ)' \bigdot W^\circ]
\right) = {\mathcal {E}}(- (U^\circ)' \bigdot W^\circ )^{-1}.
\end{align*}
}

By lemma~\ref{tildev},
\begin{align*}
{\mathcal {E}}(- (U^\circ)' \bigdot W^\circ )_0 &=   \frac{\Phi_0}{\Phi_T} \frac{\tilde v(0, X_0)}{\tilde v(T, X_T)} \exp\left(\int_0^T \psi(t, X_t, \nu_t) \, \dd t\right).
\end{align*}
Combining the preceding displays gives \eqref{likelihood}.
It remains to show equivalence between $\PP$ and $\PP^\circ$. By equivalence of $\PP$ and $\tilde \PP$,
it suffices to show equivalence between $\tilde \PP$ and $\PP^\circ$.
Under $\tilde\PP$, 
$\hat \nu_t =  \tilde\EE\, \nu_t$ is bounded and
$\nu - \hat\nu$ is a centred Gaussian process on the separable Banach space $C([0,T], \RR^d)$ with supremum norm $\|\cdot\|_\infty$, so by Fernique's theorem since $P_T$ is finite, see
\cite[theorem 2.7]{Da_Prato_Zabczyk_2014}, there is $\epsilon > 0$ such that
\begin{equation}\label{fernique}
\tilde \EE \exp(\epsilon \|\nu -\hat \nu\|_\infty/2) < \infty
\end{equation}
and therefore
\[
\tilde \EE \exp(\epsilon \|\nu \|_\infty/2) 
\le \tilde \EE \exp(\epsilon \|\nu -\hat \nu\|_\infty/2 + \epsilon\| \hat \nu\|_\infty/2) < \infty.
\]
Then
\[
 \cE((U^\circ)' \bigdot W)_t = \cE\left(\int_0^t \left( \sigma'(s, X_s)  P_s^{-1} (\nu_s - X_s)\right)' \, \dd W_s\right) 
\]
and
\begin{align*}    
& \tilde \EE \exp\left(\frac12\int_t^{\min(t+\delta, T)} \|\sigma'(s, X_s)  P_s^{-1} (\nu_s - X_s)\|^2 \, \dd s \right)\\
& \qquad \le\tilde \EE  \exp\left(\frac12\bar a \bar p \delta\left(\|\nu - \hat\nu\|^2_\infty + \|\hat\nu\|^2_\infty +  \|   X\|^2_\infty \right)\right)\\
& \qquad \le \exp\left(\frac12\bar a \bar p \delta  \|\hat\nu\|^2_\infty\right) \cdot \tilde \EE  \exp\left(\frac12\bar a \bar p \delta\|\nu - \hat\nu\|^2_\infty\right)  \cdot \tilde \EE \exp\left(\frac12\bar a \bar p \delta  \|   X\|^2_\infty \right),
\end{align*}
where $\bar a = \sup_{t \in [0,T], x \in \RR^d} {\spectral}\sigma(t,x)\sigma(t,x)'{\spectral}$ and $\bar p = \sup_{t \in [0,T]} {\spectral}P_t^{-2}{\spectral}$ where ${\spectral} \cdot {\spectral}$ denotes the spectral norm. In the last step, we have used that $\hat \nu$ is deterministic and $\nu$ and $X$ are independent under~$\tilde \PP$.

Therefore, we may choose $\delta > 0$ small enough, such that by \eqref{fernique} and assumption~\ref{novikov-ass}, the preceding display is finite and bounded in $t \in [0,T]$. This establishes that for each $t \in [0,T]$  Novikov's condition holds for $(U^\circ)' \bigdot W$ for a short time $\delta$ independent of $t$, showing equivalence between $\tilde \PP$ and $\PP^\circ$. 
\end{proof}

Note the parallels between \eqref{likelihood} and formula (1.11) in \cite{Mider2021} for the discretely observed case.

In the proof we used the following lemma, which is proven in appendix~\ref{app:proofs} directly from proposition \ref{backwardkalmanbucy} using It\^o's formula. Heuristically, the result follows from remark~\ref{palmowski}.
\begin{lemma}\label{tildev} In the setting of theorem~\ref{girsanov},
\begin{align*}
\dd\log \tilde v(t, X_t) &=     
 -  (H_t X_t)' \, \dd Y_t + \frac12 \|  H_t  X_t \|^2 \, \dd t \\
&\quad + (\sigma'(t, X_t) P_t^{-1} (\nu_t-X_t))' \, \dd W^\circ_t  +\frac12 \|\sigma (s, X_t) P^{-1}_t (X_t-\nu_t)\|^2 \, \dd t
\\
&\quad +  (b(t, X_t)   - (B_tX_t + m_t)  )' P_t^{-1} (\nu_t-X_t)\, \dd t  \\
&\quad- \frac12\operatorname{Tr}((a(t, X_t) - \tilde a_t) P_t^{-1}) \, \dd t + \frac12 ( X_t- \nu_t)'  P^{-1}_t (a(t, X_t) - \tilde a_t) P^{-1}_t (X_t-\nu_t) \, \dd t.
\end{align*}
\end{lemma}

\subsection{Monte Carlo method}\label{cranknicolson}
The last theorem~\ref{girsanov} can be used directly to estimate the expectation of functionals of the conditional process. It would be natural to derive estimates by Monte Carlo sampling, but if the Monte Carlo weights $\frac{\dd \P_{X \mid Y= y}}{\dd P^\circ}(X^\circ)$ have too high variance, it can be better to use instead Markov Chain Monte Carlo methods for example. Here we give the Metropolis--Hastings sampler with an autoregressive random walk on the driving Brownian motion as example. In line with the spirit of the paper, we formulate the Metropolis--Hastings algorithm in  the space of sample paths, see \cite{10.1214/aoap/1027961031}, and do not discuss approximation errors. In our implementation in section~\ref{sec:simulations}, we use a fine time grid to make sure that the approximation error of the SDEs is small compared to the Monte Carlo sampling error. 

In order to apply the Metropolis--Hastings algorithm, we need to define the proposal kernel~$Q$, a measurable function~$\alpha$ for the acceptance probabilities and the target distribution~$\pi$.
On Wiener space equipped with the Wiener measure $\P_W$, we define the proposal kernel
\[
Q(w, \dd w') 
\]
indirectly, by conditional on the current Brownian path~$w$ proposing a new Brownian path 
\begin{equation*}
    w' = \varrho w + \sqrt{1-\varrho^2} \, \bar W, 
\end{equation*}
where $\bar W$ is an independent Brownian motion and $\varrho \in (0,1)$ is an autocorrelation parameter governing the size of random walk steps in Wiener space.
We denote the strong solution of~\eqref{guided}, whose existence we assume, by $F^\circ$ such that
\[
X = F^\circ(W^\circ).
\]
The target distribution $\P_{X \mid Y = y}$ has density $\frac{\dd \P_{X \mid Y= y}}{\dd \P^\circ}$, known up to a constant, with respect to the ``reference measure'' $\P^\circ$.
Thus, to sample from $\P_{X\mid Y=y}$, we can sample from the measure $\pi$ defined via its $\P_W$ density, 
\[
\frac{\dd\pi}{\dd \P_W}  = \frac{\dd \P_{X \mid Y= y}}{\dd \P^\circ}\circ F^\circ,
\]
which has $\P_{X\mid Y=y}$ as push forward under $F^\circ$, such that $F^\circ(Z) \sim \P_{X\mid Y=y}$, where $Z \sim \pi$. To sample from $\pi$ we are using Metropolis--Hastings with proposal kernel $Q$.
To be precise, given $z$, we propose
$z' = \varrho z + \sqrt{1-\varrho^2} \, \bar W$ and
accept the proposed $z'$ with probability $\alpha(z, z')$ as new sample $z$,  otherwise we reject the proposal and retain the old sample $z$. Here,
\[
\alpha(z, z') = \min\left(
\frac{\frac{\dd\pi}{\dd \P_W} (z')}
{\frac{\dd\pi}{\dd \P_W} (z)
}, 1
\right)
\]
is the Metropolis-Hastings acceptance rate.

In appendix~\ref{app:proofs} we prove the following proposition that this procedure leads to the correct stationary distribution.
\begin{proposition}\label{balance}
     The measure~$\pi$ is the stationary distribution of the Metropolis--Hastings chain with proposal kernel $Q$ and acceptance probability $\alpha$. 
\end{proposition}

\subsection{Bellman principle and optimal control for the smoothing distribution}\label{bellman}

Finding the smoothing distribution can be thought of as a stochastic control problem, see e.g.~\cite{TzenRaginsky}. Hence, in this section, we revisit heuristically the process defined in \eqref{guided} from a stochastic control perspective, this also elucidates \eqref{parabolic2}. 
Just as under additional regularity assumptions $V_t = v(t, X_t)$, we can make the ansatz $U_t = u(t,X_t)$ for a choice of (sufficiently smooth) $u: [0,T]\times \RR^d \to \RR^d$.
Here we consider to choose $u$ which minimizes a Kullback--Leibler type cost function:
Let $\PP^ u$ defined by the change of measure $\frac{\dd \PP^u}{\dd \PP} =
\cE\left(U'\bigdot W \right) 
= \cE\left(-U'\bigdot  W^u\right)^{-1} = (\frac{\dd \PP}{\dd \PP^u})^{-1}$.

We define the reward function (the negative of a cost function)
\begin{equation*}
J_s^t = \int_s^t (H X_r)' \, \dd Y_r  - \frac12 \int_s^t \| H X_r\|^2 \, \dd r - \int_s^t u(r, X_r)' \, \dd W^u_r - \int_s^t \frac12 \|u(r, X_r)\|^2 \, \dd r,
\end{equation*}
which is the sum of the log-likelihood process and the negative log-likelihood between $\PP^u$ under which $W^u$ is a Brownian motion and~$\PP$.
This objective can be made light of in terms of the variational formula of \cite{Donsker1983}, which relates a conditional distribution to a Kullback--Leibler optimization problem: As example, let $\PP^\star$ be the conditional or posterior measure induced having observed a datum with likelihood $Z$, using $\PP$ as prior. That is $\PP^\star$ is the measure given by $\dd \PP^\star = Z \,
\dd \PP$ using the likelihood~$Z$ as Radon--Nikodym derivative with respect to the prior. By the variational formula, a consequence of Jensen's inequality, $\PP^\star$ is the maximizer of $
\EE^\QQ \log Z   - \operatorname{KL}(\QQ \parallel \PP)
$ over measures $\QQ \ll \PP$.
Here we essentially restrict our search to random measures of the form $\QQ = \PP^u$, $ u(t, X_t)$ with
\[
\log \frac{\dd \PP^u}{\dd \PP} = \textstyle\int_0^T u(t, X_t)' \, \dd W_t - \frac{1}{2} \int_0^T \|u(t, X_t)\|^2 \, \dd t 
\]
for some process $u(t, X_t)$, 
observe $\EE^u[\log \frac{\dd \PP^u}{\dd \PP}] = \operatorname{KL}(\PP^u \parallel \PP)$,  and search for an optimal control $u^\star$ attaining the maximum of the conditional expectation. ($\EE^u$ denoting the expectation under $\PP^u$.) To solve this, we consider the optimal expected reward-to-go conditional on $Y$, that is
\[
    \log v(s, x) := \sup_u \EE^u\left[J_s^T\mid X_s = x, Y\right],
\]
and use the Bellman principle: ``An optimal policy has the property that whatever the initial state and initial decision are, the remaining decisions must constitute an optimal policy with regard to the state resulting from the first decision.'' (See \cite{bellman1957dynamic}, Chap. III.3.)
So
\[
   \sup_{u} \EE^u \left[J_s^{s + \epsilon} + \log v(s+\epsilon,X_{s + \epsilon}) \mid X_s = {x}, Y \right] - \log v(s,x)  = 0.
\]
Taking the formal derivative, 
\begin{align*}
(\dd \log v)(s, \cdot)  & + (\mathfrak L \log v)(s, \cdot) \, \dd s
+  (H_t \cdot)' \, \dd Y_s  - \frac12   \|H_t \cdot\|^2 \, \dd s \\
& +\sup_u \left(
  u'(s,\cdot)\sigma'(s,x) \pder{x} \log v(s,\cdot)
-   \frac12 \|u(s, \cdot)\|^2\right) \dd s = 0.
\end{align*}
The supremum $ \frac12 \|u(s, \cdot)\|^2$ of the right-most term is attained in \[
u(s,x) = \sigma'(s,x) \pder{x} \log v(s,x)
\]
(see \cite{TzenRaginsky}), with $v(s,x)$ solving the parabolic stochastic differential equation 
\begin{equation}\label{parabolic}
\dd \log v(t, \cdot) + \fL \log v(t, \cdot) \, \dd t  + (H_t \cdot)'\diamond \dd Y_t + \frac12\|\sigma(s,\cdot)' \nabla \log v(t, \cdot)\|^2 \, \dd t - \frac12 \|(H_t \cdot)'\|^2 \, \dd t = 0.
\end{equation}

We now have the Hamilton--Jacobi--Bellman stochastic parabolic differential equation for $\log v$ and~$v$.
We can relate \eqref{parabolic} to \eqref{parabolic2} by remark~\ref{palmowski}.
By the verification theorem (c.f.~\cite{TzenRaginsky}), the reward-to-go $\log v$ characterizes the optimal control $u^\star$: If $v$ is the solution with boundary condition $\log v(T, \cdot) \equiv 0$, then
\[
u^\star(t,x) = \sigma(t,x)' \nabla \log v(t,x)
\]
is the optimal control and the drift of $X$ under~$\PP^\star$ is $b + a \nabla \log v$.

We can therefore think of \eqref{conditional} as solution to a Kullback--Leibler control problem.

\subsection{Optimizing the guide}\label{adam}

The performance of the method depends on the choice of linear auxiliary process, see the discussion  in \textcite{schauer2017guided} (Section 4.4) for general ways to approach this. 
Let $\theta = \operatorname{vec}(B, m, \sigma)$ be the tuneable parameters, determining $X^\circ$ using a time-homogenous linear process as guide. We write  
 $\P^\circ_\theta$ to emphasize the dependence of the law of  $X^\circ$ on the choice of $\theta$.
 In line with section~\ref{bellman} we propose to simulate the process $X^\circ$ and the Jacobian process $DX^\circ$, see \cite{RogersWilliams2}, \cite{pmlr-v118-li20a}, 
 \[(DX^\circ_t)^{(i,\cdot)} = \partial_{\theta_i} X^\circ_t, t \in [0, T], \quad i \in \{1, \dots, d\},
 \]
 using a differentiable numerical solver and to maximize the reward
 \begin{equation}\label{reward2}
J = \int_0^T (H_t X^\circ_t)' \, \dd Y_t - \frac12 \int_0^T \|H_t X^\circ_t\|^2  \, \dd t  + \log L(X^\circ_T; \zeta)  - \int_s^t u(r, X_r)' \, \dd W_r - \int_s^t \frac12 \|u(r, X_r)\|^2 \, \dd r,
 \end{equation}
 using gradient based optimization methods with gradients $\nabla_\theta J$.

\section{Application: Reaction-diffusion}\label{sec:simulations}
We consider the following discretization of a stochastic reaction diffusion system as a $d = 100$ dimensional stochastic differential equation
\begin{equation}\label{reaction1}
\dd X_t = -5\Lambda X_t \, \dd t + F(X_t) \, \dd t + \dd W_t, \quad X_0 = 0 ,
\end{equation}
for $t \in [0,1]$, where $W = (W_t)$ is an $\RR^d$-valued Brownian motion. The tridiagonal matrix $\Lambda \in \RR^{d \times d}$ is given as
\begin{equation*}
\Lambda^{(1,1)} = \Lambda^{(d,d)} = 1, \quad \Lambda^{(i,i)} = 2, \quad 1 < i < d, \quad \Lambda^{(i, i+1)} = \Lambda^{(i+1, i)} = -1, \quad 1 \le i < d
\end{equation*}
and   a non-linear force 
\begin{equation*}
F(x)_i = 2x_i - 2x_i^3, \quad x \in \RR^d,
\end{equation*}
that causes spontaneous organization.

This system corresponds to a semidiscrete finite difference approximation in space of the stochastic reaction-diffusion equation
\begin{equation*}
    \dd \mathbb X(t,x) = \Bigl(5 \Delta   \mathbb X(t,x) + 2 \mathbb X(t,x) - 2\mathbb X^3(t,x)\Bigr) \, \dd t + \dd \mathbb W(t,x)
\end{equation*}
with initial condition $\mathbb X(0,x) = 0$ and truncated to a finite space interval of~$\RR$ with grid size~$1$ and Neumann boundary conditions. Here $\Delta = \partial_{xx}$ denotes the Laplace operator, and $\mathbb W$ is a cylindrical Wiener process.

\begin{figure}[htp] 
    \centering
    \includegraphics[width=0.49\linewidth]{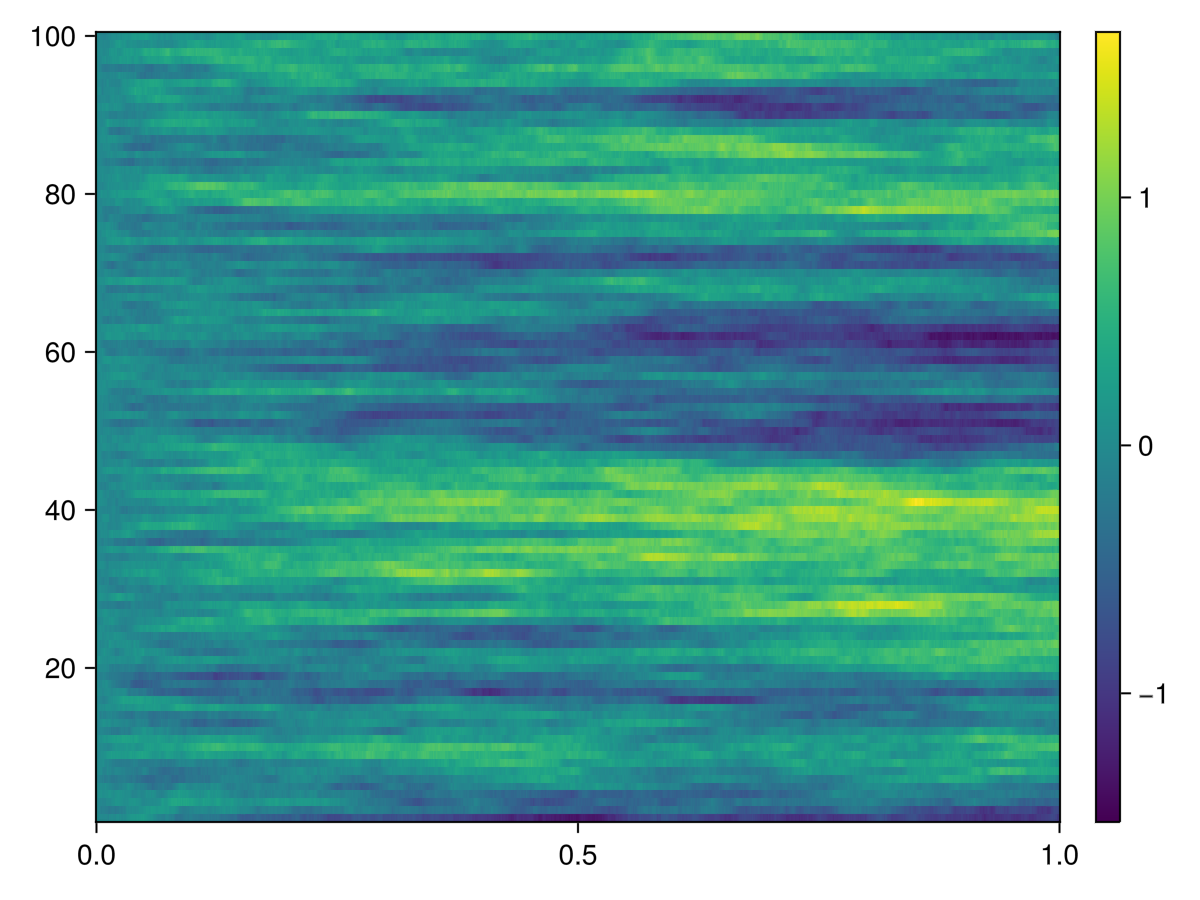}\includegraphics[width=0.49\linewidth]{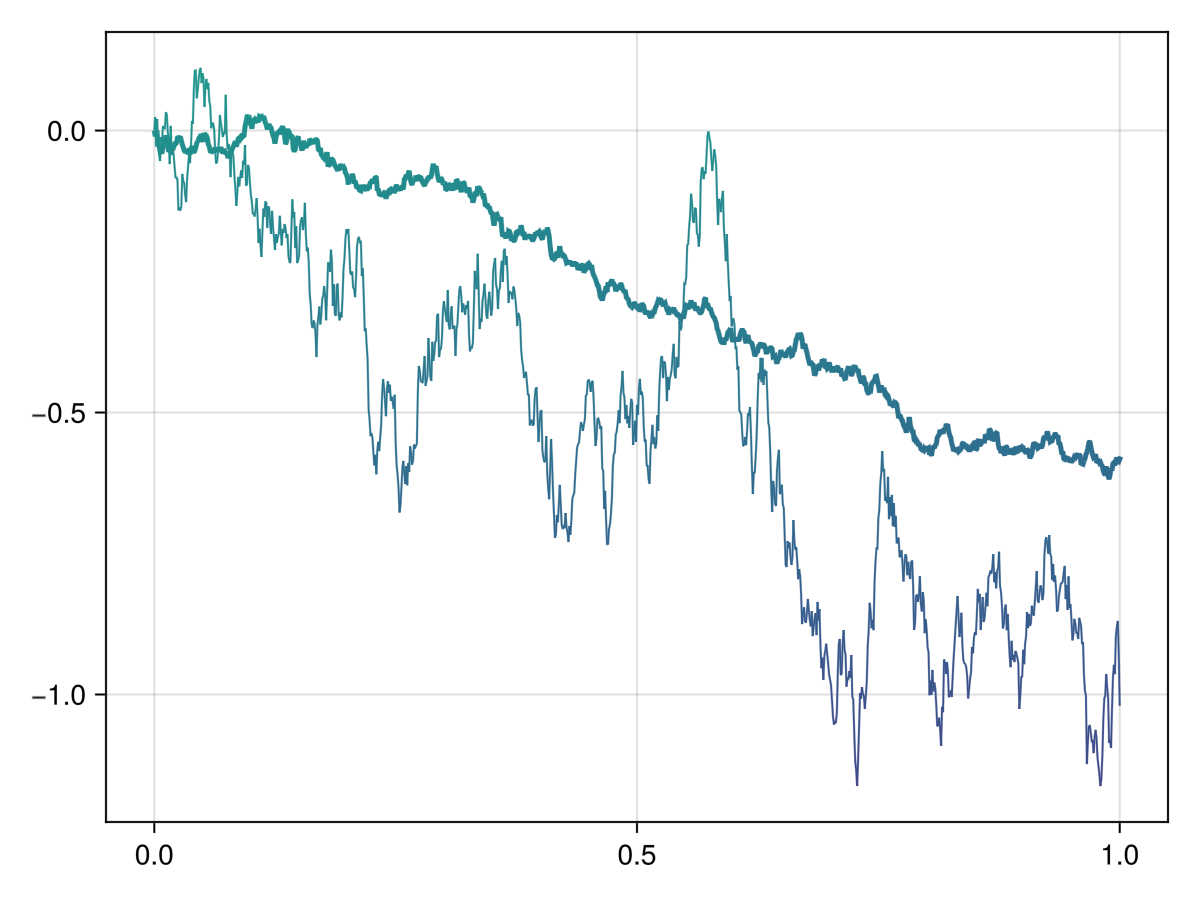}
    \caption{Forward simulation of system \eqref{reaction1}. Left: Heatmap of $(X^{(j)}_{t})$. (Horizontal axis: time. Vertical: Coordinates. Color: Value.) Right: Coordinate process $X^{(50)}$ and, in bold, the corresponding observation process $Y^{(50)}$.}
    \label{fig:forward}
\end{figure}
We observe the process
\[
Y_t = \int_0^t H X_s \, \dd s  +  \beta_t, \quad Y_t =  0, 
\]
with $H = 5$ as well as the noisy final value $\zeta = X_1 + Z$, $Z \sim N(0, 0.1 I_d).$ 

We implemented our procedure for this model in Julia, \cite{Julia-2017}. All code is available from \url{https://github.com/mschauer/GuiSDE.jl}. 
The Euler--Maruyama scheme with step size 0.001 is employed to generate a sample of~$X$ as ground truth and corresponding observation~$Y$.
Figure~\ref{fig:forward} shows the process~$X$ on the left and coordinate processes $X^{(50)}$, $Y^{(50)}$ on the right. 

We approximate $X$ by the linear process $\tilde X = (\tilde X_t, t\in[0,1])$ given by \begin{equation*}
\dd \tilde X_t = -5\Lambda \tilde X_t \, \dd t + \dd W_t, \quad \tilde X_0 = 0.
\end{equation*}
Given our observation~$Y$, we solve the backward system from proposition~\ref{backwardkalmanbucy} to obtain $U^\circ$, see~\eqref{Ucirc}.

We use the autoregressive random walk Metropolis--Hastings from section~\ref{cranknicolson} with 5000 iterations to compute an estimate of the process $\mu^\star _t = \EE [X_t \mid Y, \zeta]$. Figure~\ref{fig:smooth} shows the obtained estimate as well as a sample of $X \mid Y, \zeta$. 
\begin{figure}
    \centering
    \includegraphics[width=0.49\linewidth]{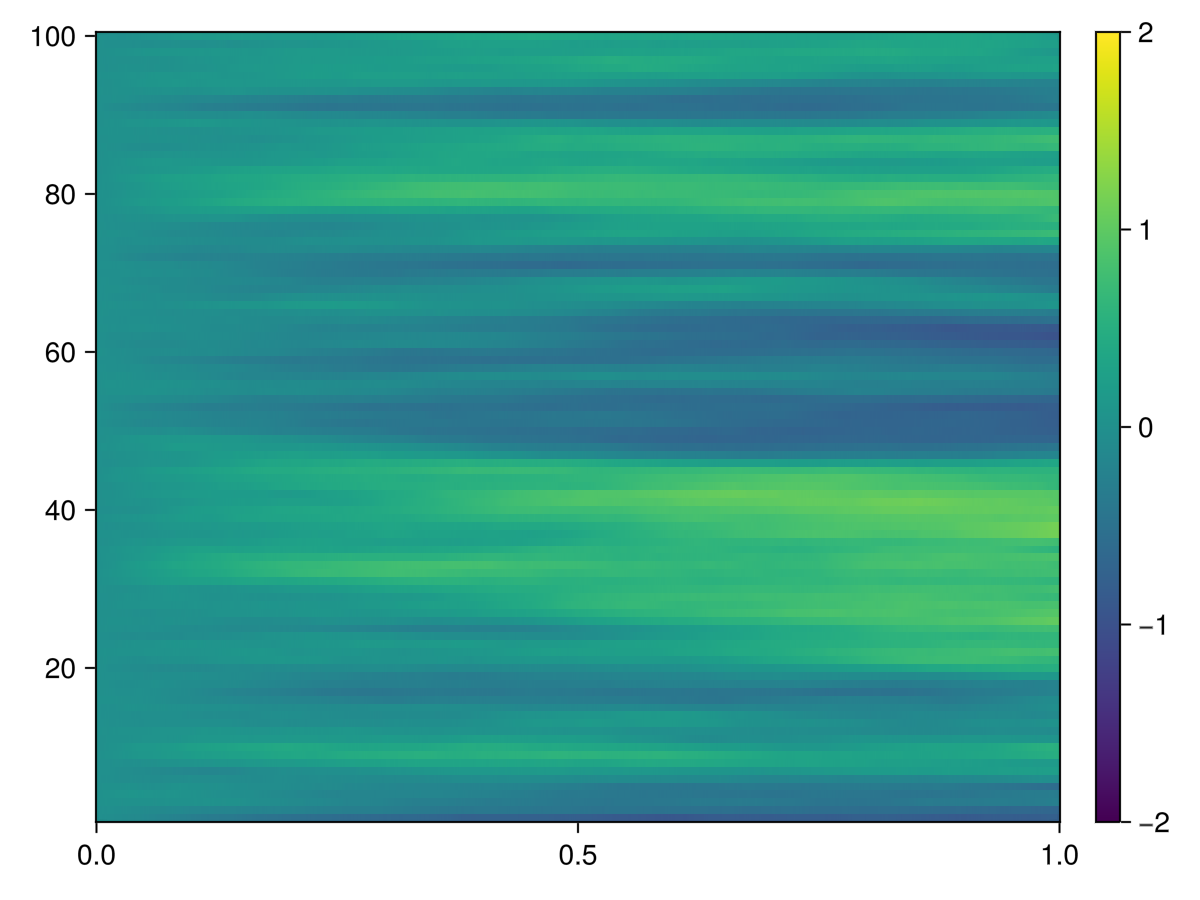}\includegraphics[width=0.49\linewidth]{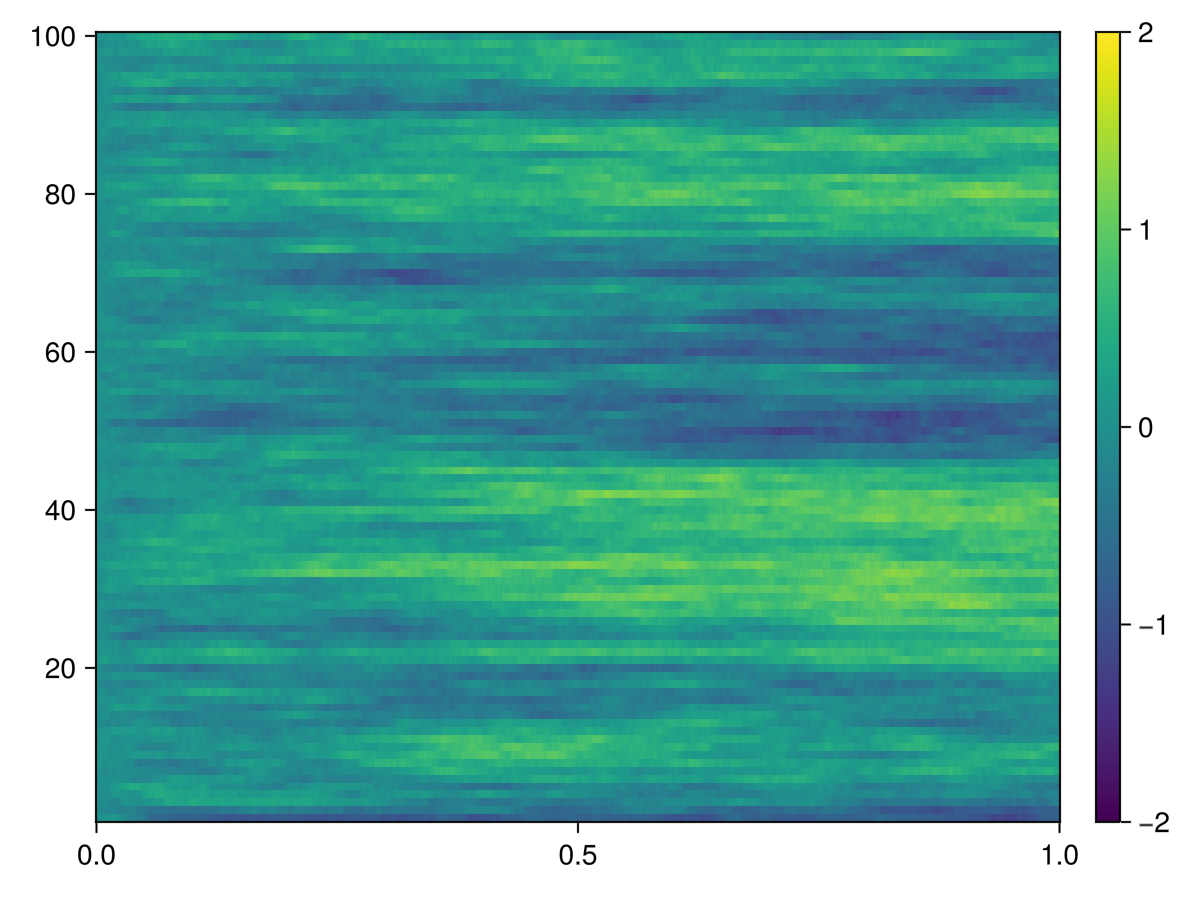}
    \caption{Left:  Estimate $\EE [X_t \mid Y, \zeta]$. Right: A sample of $X \mid Y, \zeta$.}
    \label{fig:smooth}
\end{figure}

Finally, we fit a variational approximation to the conditional law $X \mid Y, \zeta$ using the approach of section~\ref{adam} on the system. Here we take the same parameters, except the dimension is $d=20$, to allow the use of (simpler) forward mode automatic differentiation. The results are shown in figure~\ref{fig:variational}.

For $B$ we consider all symmetric, tridiagonal matrices, parametrized as
\begin{equation*}
B_t^{(i,i)} = \theta_i, \quad 1 \le i \le d,  \quad B_t^{(i, i+1)} = B_t^{(i+1, i)} = \theta_{d+i}, 
\quad 1 \le i < d
\end{equation*}
and $m_t$ we parametrize as $m_t^{(i)} = \theta_{ 2d + i-1}$.
We denote by $\theta^\circ$ the ad hoc choice $B(\theta^\circ) = -5\Lambda$ from earlier. 
We use the Julia package ForwardDiff (\cite{RevelsLubinPapamarkou2016}) to compute gradients and the Julia package Optimisers from Julia's Flux ecosystem (\cite{Flux.jl-2018}) for the optimisation. 

We run Adam with learning rate parameter $\eta = 0.01$ and starting value $\theta=0$ for $2000$~iterations to find $\theta^\star$ maximizing the expected reward~\eqref{reward2}. 
The training took 187\,min on a single Intel Xeon Platinum 8180 CPU at 2.50GHz.

\begin{figure}
    \centering
    \includegraphics[width=0.32\linewidth]{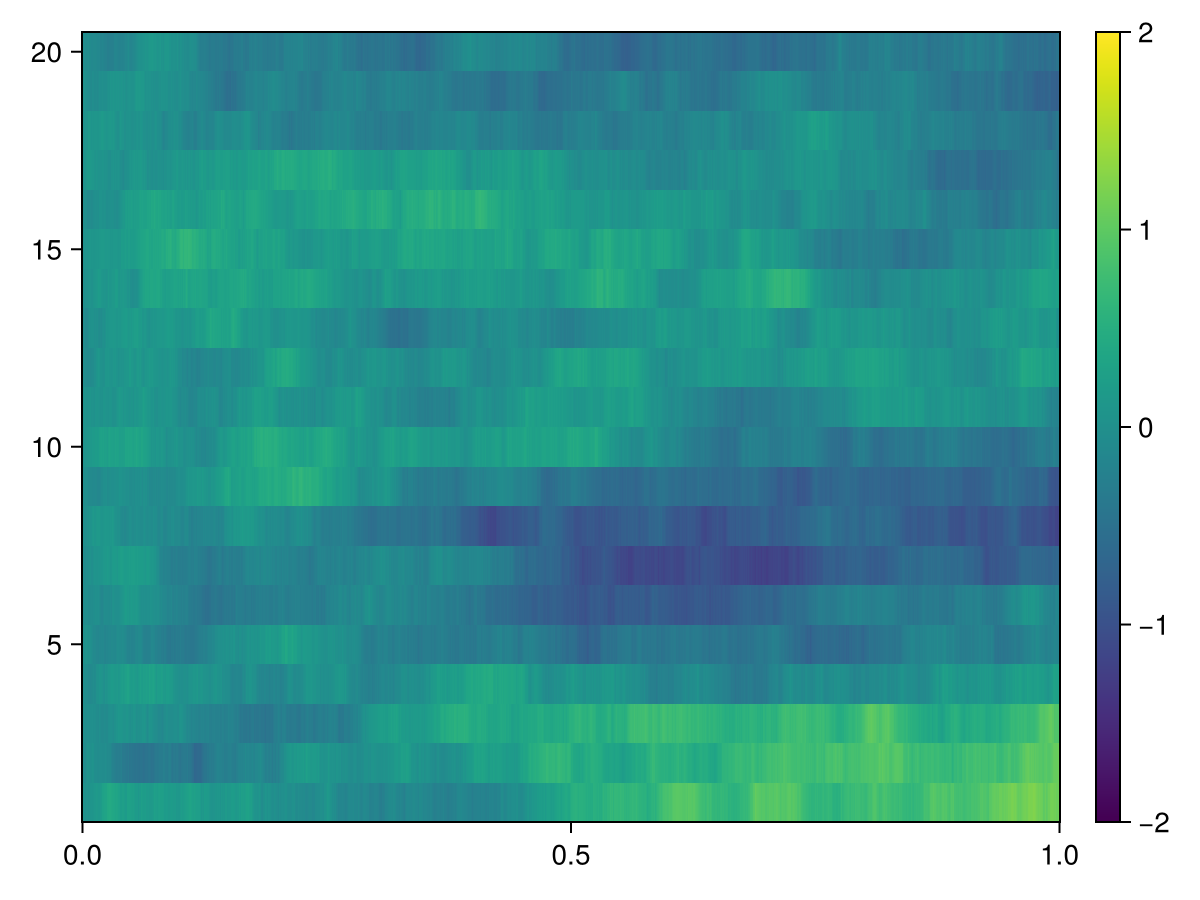}
    \includegraphics[width=0.33\linewidth]{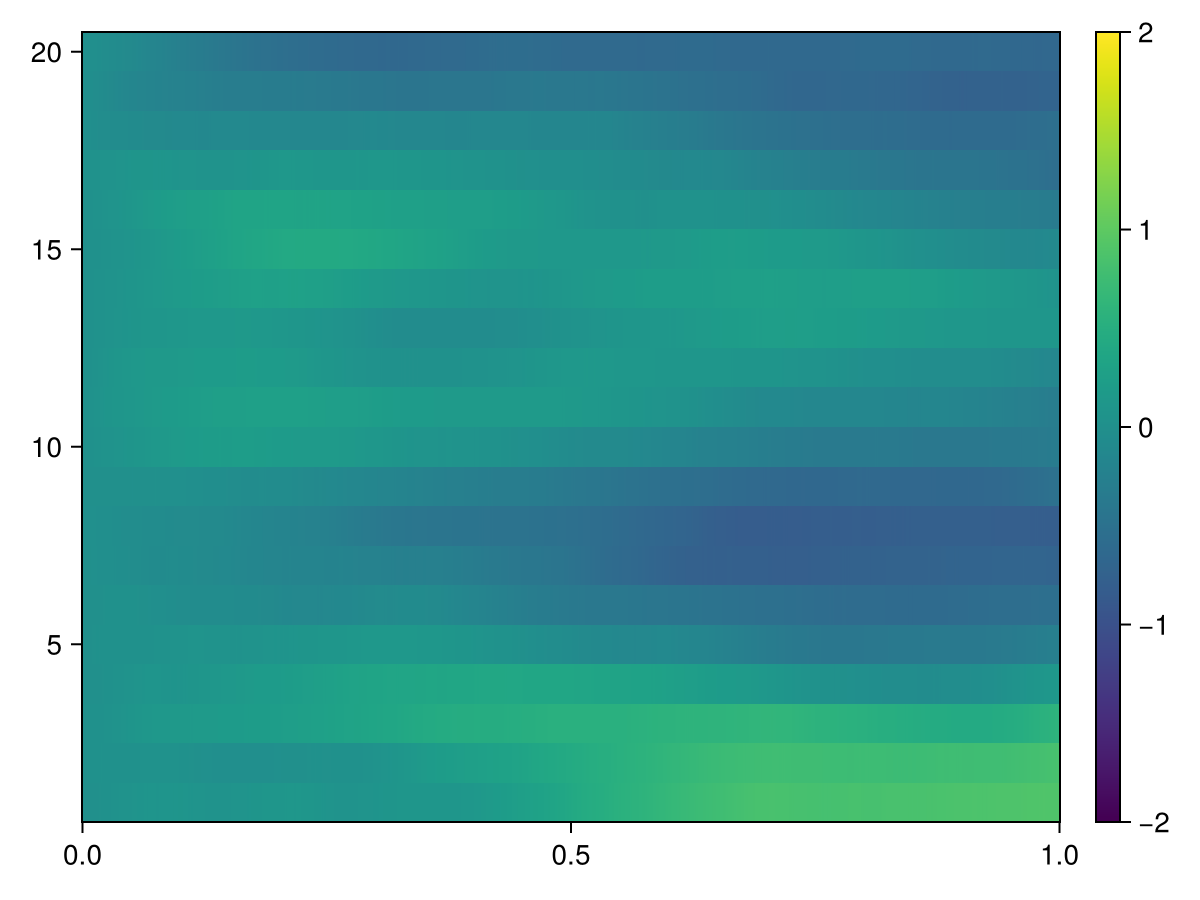}\includegraphics[width=0.33\linewidth]{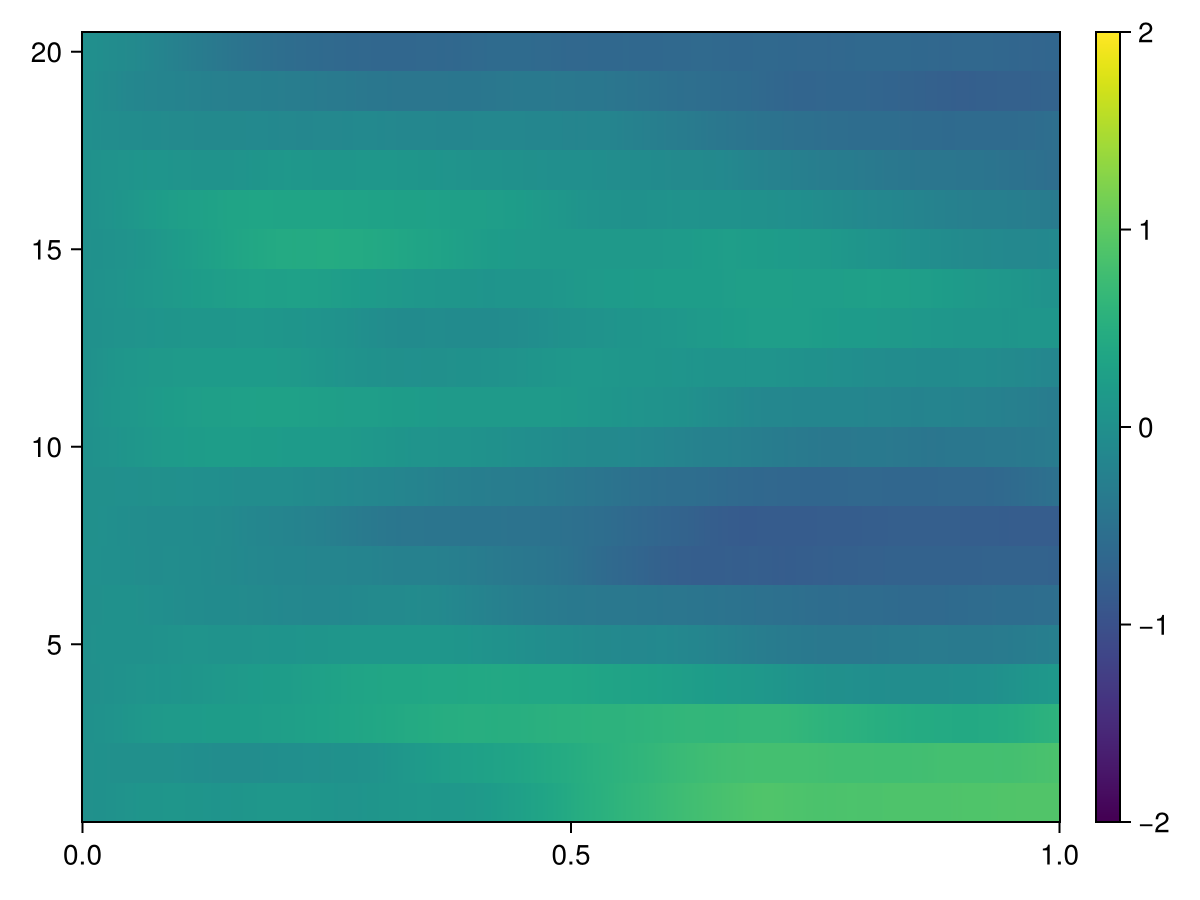}
    \caption{Left: Forward simulation of the system \eqref{reaction1}, $d=20$. Centre: Posterior mean computed with Metropolis--Hastings for a noisy observation of the left sample. Right: Mean of variational approximation of the same posterior.\label{fig:variational}}
\end{figure}

Figure~\ref{fig:training} shows the training loss curve and the entries of the matrix $B(\theta^\star)$, and for comparison, the matrix $-5\Lambda$, as heatmap.
\begin{figure}
    \centering
    \includegraphics[width=0.3\linewidth]{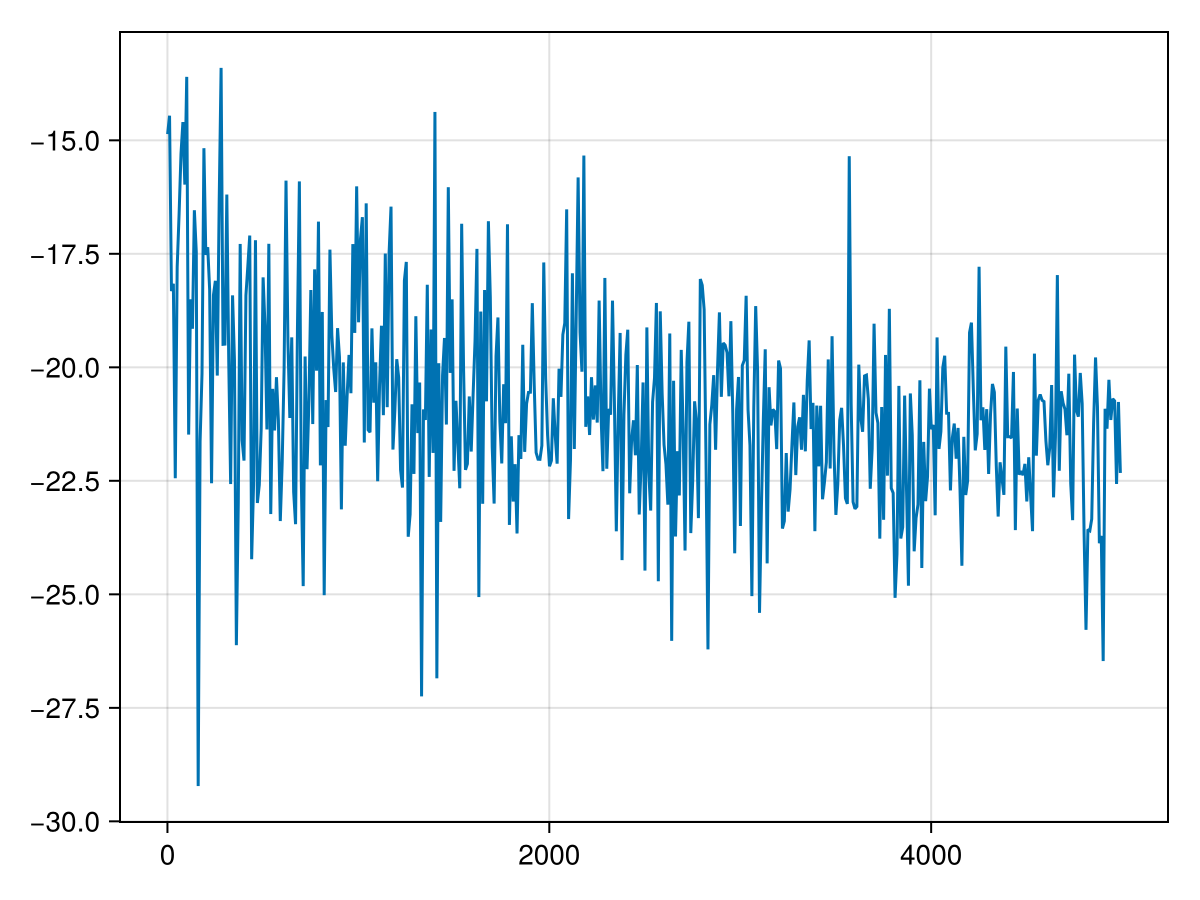} \includegraphics[width=0.3\linewidth]{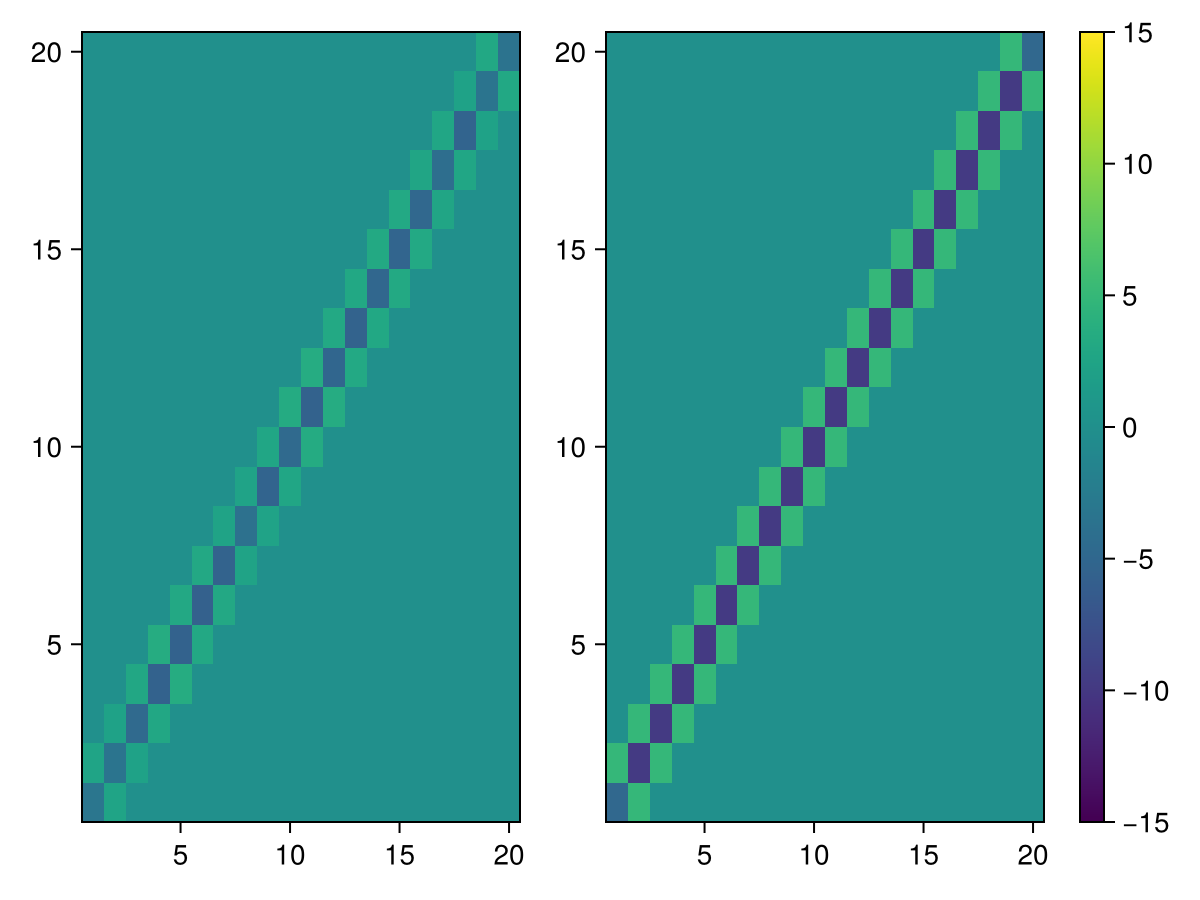}\includegraphics[width=0.3\linewidth]{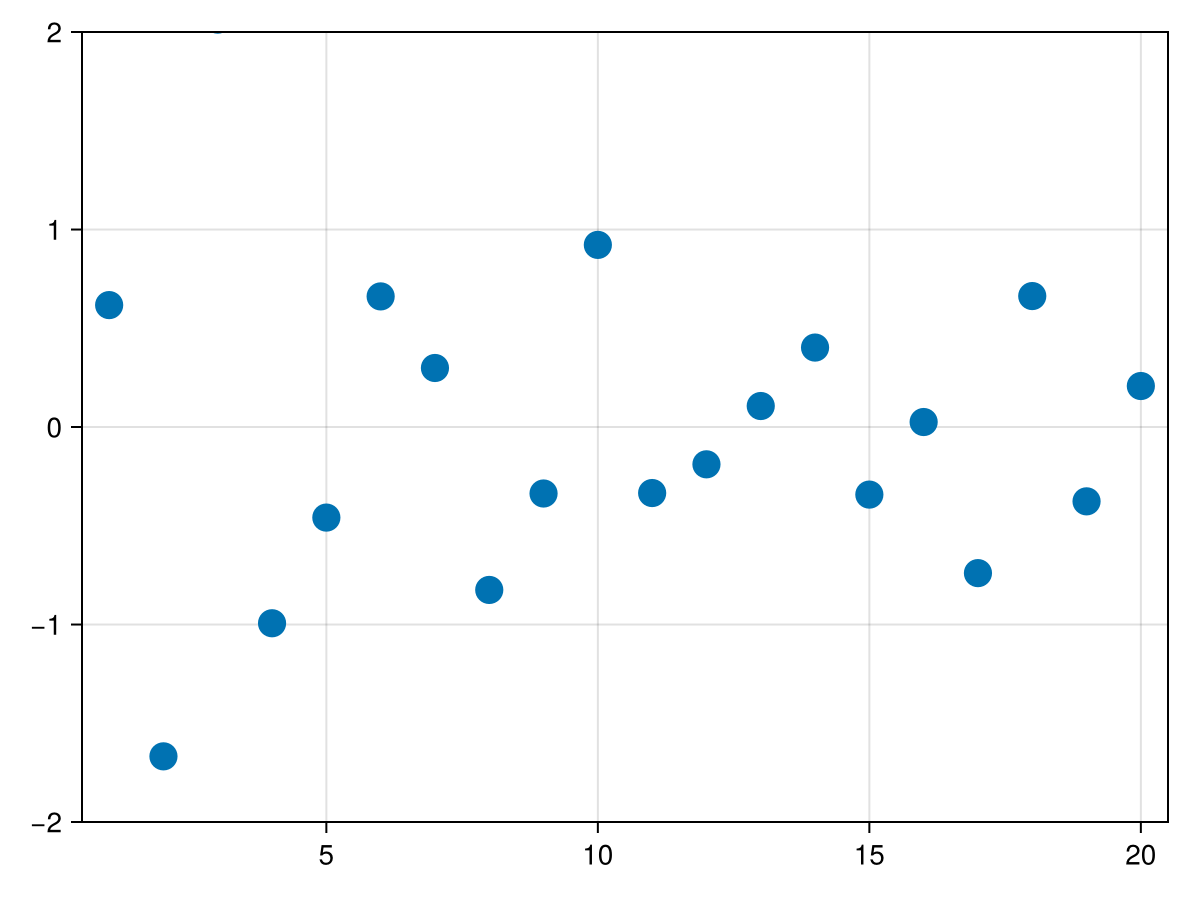}
    \caption{Left:  Loss curve (negative reward versus iteration). Center: $B(\theta^\star)$ versus   $-5\Lambda$. Right: $m(\theta^\star)$.}
    \label{fig:training}
\end{figure}
Our procedure finds a decent approximation to the posterior. 
To compare, we compare both $\hat\mu_t =  \EE^\circ_{\theta^\star} [X_t \mid Y, \zeta]$ and $\mu^{\circ} _t =  \EE^\circ_{\theta^\circ} [X_t \mid Y, \zeta]$ with $\mu^\star$, the posterior mean. Here $\PP^\circ_{\theta}$ is the $\PP^\circ$ measure corresponding to the choice~$\theta$ of the variational parameter.
While $\mu^{\circ}$ is already quite close to $\mu^\star$, $\hat\mu$ is much closer, $\mu^\star$ see figure~\ref{fig:error}.  
\begin{figure}
    \centering
    \includegraphics[width=0.49\linewidth]{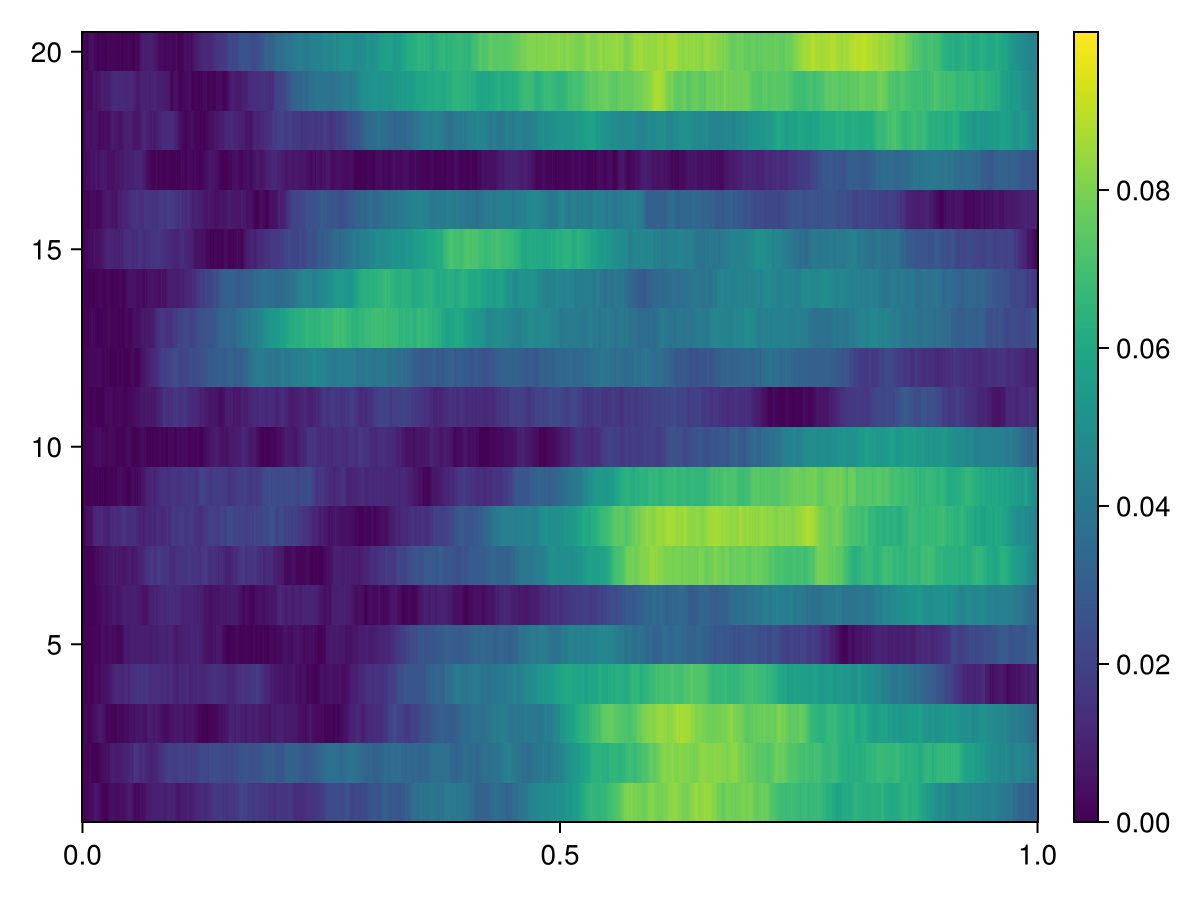}\includegraphics[width=0.49\linewidth]{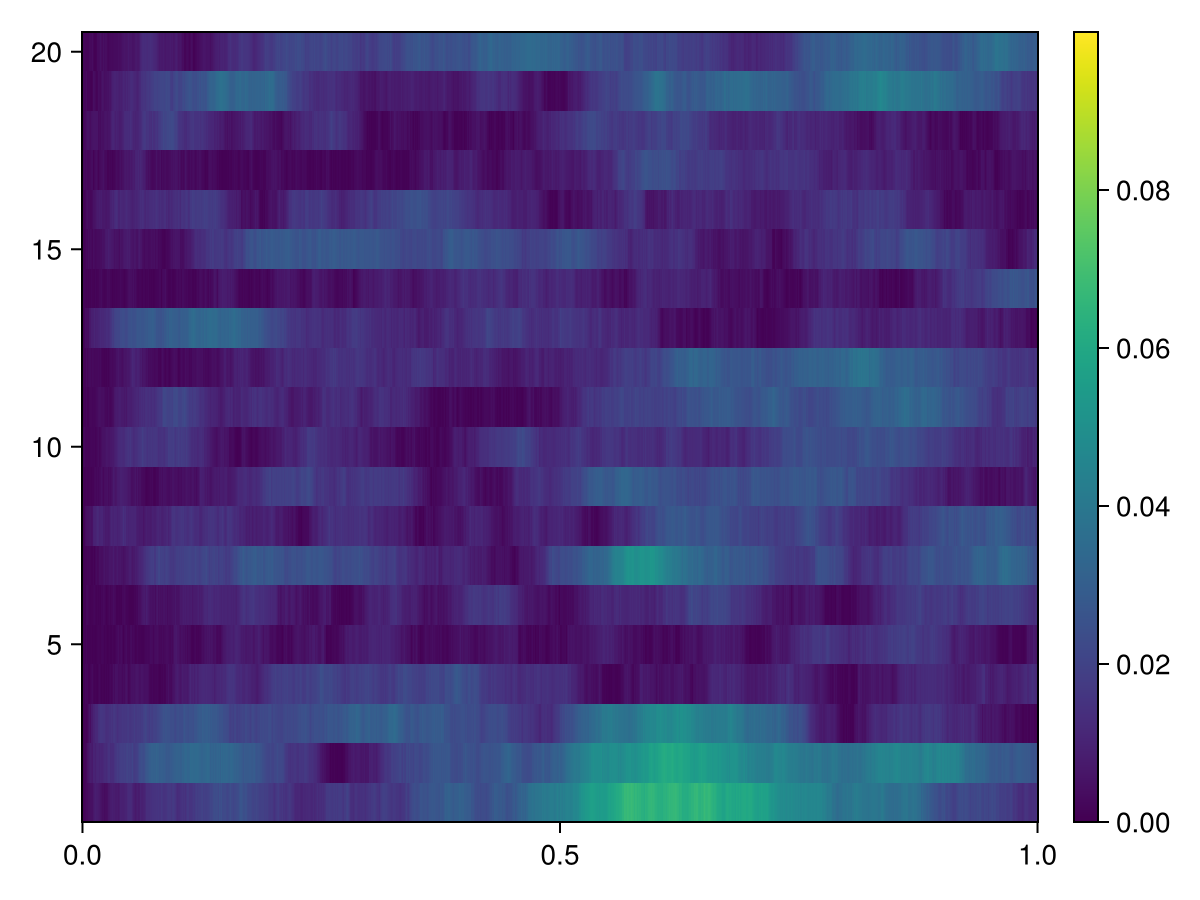}
    \caption{Left:  Error $|\mu^{\star,(i)}_t-\mu^{\circ,(i)}_t|$ corresponding to the ad hoc choice. Right: Error $|\mu^{\star,(i)}_t-\hat\mu^{(i)}_t|$ corresponding to the variational optimum.}
    \label{fig:error}
\end{figure}

\subsection*{Acknowledgements}
The authors like to thank Kasper Bågmark for helpful discussions.
The work presented in this article was supported by the Chalmers AI Research Centre (CHAIR) through the project ``Stochastic Continuous-Depth Neural Network'', by the European Union (ERC) through grant no.~101088589, ``StochMan'', and by the Swedish Research Council (VR) through grant no.~2020--04170. Views and opinions expressed are however those of the authors only and do not necessarily reflect those of the European Union or the European Research Council Executive Agency. Neither the European Union nor the granting authority can be held responsible for them.

\printbibliography

\appendix

\section{Proofs}
\label{app:proofs}

\begin{proof}[Proof of corollary~\ref{parametrized}]
By Clark's formula, \cite[IV.41]{RogersWilliams2}, cited after \cite{2637511},  and the assumed independence,  
\begin{align*}
[p, W]_t & = [ p(\cdot, X_\cdot, \Theta_\cdot), W]_t  =  \int_0^t (\Phi_t \nabla_x v(s, X_s, \Theta_s))' \dd [X]_s \\
&= \int_0^t (\sigma'(s, X_s)\nabla_x \log v(s, X_s, \Theta_s))' v(s, X_s, \Theta_s) \dd s. \qedhere
\end{align*}
\end{proof}

\begin{proof}[Derivation of remark~\ref{palmowski}]
Rewrite for clarity the stochastic PDE~\eqref{parabolic2}
\[
\dd v(t, \cdot)  = -\fL v(t, \cdot) \, \dd t - v(t, \cdot) (H_t \cdot)'  \diamond \dd Y_t. 
\]
By It\^o's formula for stochastic backward integrals 
\[
\dd \log v(t, \cdot)  = \frac{1}{v(t, \cdot)} \, \dd v(t, \cdot) + \frac12\|(H_t \cdot)\|^2 \, \dd t. 
\]
Plugging in $\dd v$ yields
\[
\dd \log v(t, \cdot)  = -\frac{1}{v(t, \cdot)} \fL v(t, \cdot) \, \dd t -  (H_t \cdot)'  \diamond \dd Y_t + \frac12\|H(\cdot)\|^2 \, \dd t 
\]
and
\[
  \dd \log v(t, \cdot) + \fL \log v(t, \cdot) \, \dd t +\frac12\|\sigma'\nabla \log v(t, \cdot))\|^2 \, \dd t +  (H_t \cdot)'  \diamond \dd Y_t - \frac12\|(H_t \cdot)\|^2 \, \dd t = 0
\]
Lastly, $(H_t \cdot)'  \diamond \dd Y_t  = (H_t \cdot)' \, \dd Y_t$.
\end{proof}

\begin{proof}[Proof of proposition~\ref{backwardkalmanbucy}.]
This can be proven via~\eqref{parabolic2} by choosing the ansatz \eqref{linearansatz},
so that 
\[
C_t = \int v(t, x) \, \dd x,\quad 
\nu_t = \frac{1}{C_t}\int x v(t, x) \, \dd x, \quad \text{ and }\quad 
P_t = \frac{1}{C_t}\int (x - \nu_t)(x - \nu_t)' v(t, x) \, \dd x. 
\]
We only derive the stochastic differential equation for the process $C$, the equations for $\nu$ and $P$ are derived in \cite{FraserPotter1969}.

Note now that
\[
\nabla_x v(t,x) =  P_t^{-1}(\nu_t - x) v(t,x)\]
and
\[\frac12 \sum_{i,j=1}^n a_{t, ij} \frac{\partial^2}{\partial x_i \partial x_j} v(t, x) =  \frac12 \left(\|\sigma_t'P^{-1}_t(\nu_t - x)\|^2 - \operatorname{Tr}(a_t' P^{-1}_t)\right)v(t,x).
\]
Furthermore, we observe that
\[
\int  \frac12 \left(\|\sigma_t'P^{-1}_t(\nu_t - x)\|^2 - \operatorname{Tr}(a_t P^{-1}_t)\right)v(t,x) \, \dd x = 0
\]
and
\[
\int (B_t x + m_t)' P_t^{-1}(\nu_t - x) v(t,x) \, \dd x = -\operatorname{Tr}(B_t) C_t.
\]
So, plugging everything into~\eqref{parabolic2} and using stochastic Fubini to integrate over $x$, we obtain as equation for~$C$
\[
\dd C_t =  \operatorname{Tr}(B_t) C_t \, \dd t -  C_t (H_t \nu_t)' \diamond \dd Y_t 
\]
and therefore
\[
\dd \log C_t = \operatorname{Tr}(B_t) \, \dd t -  (H_t \nu_t)'  \diamond \dd Y_t + \frac12 \|H_t \nu_t\|^2 \, \dd t . \qedhere
\]
\end{proof}

\begin{proof}[Proof of lemma~\ref{tildev}]
The result follows from applying It\^o's formula and collecting terms. 
{\small
\begin{align*}
\dd\log \tilde v(t, X_t) =& +\dd\log C_t - \tfrac12\dd \log |P_t| + \dd \left(-\tfrac12(X_t-\nu_t)' P_t^{-1} (X_t-\nu_t)\right)\\
 =&   \operatorname{Tr}(B_t) \, \dd t -  (H_t \nu_t)' \diamond \dd Y_t + \tfrac12 \|H_t \nu_t\|^2 \, \dd t  - \tfrac12\operatorname{Tr}((\pder t P_t) P_t^{-1}) \, \dd t\\
& -(\dd X_t-\dd \nu_t)' P_t^{-1} (X_t-\nu_t) -(\dd X_t-\dd \nu_t)' P_t^{-1} (\dd X_t - \dd \nu_t)/2\\
& -\tfrac12( X_t- \nu_t)' (\pder t P_t^{-1}) (X_t-\nu_t) \, \dd t \\
= &-  (H_t \nu_t)' \, \dd Y_t + \tfrac12 \|H_t \nu_t\|^2 \, \dd t -  (H_t \dd \nu_t)' \, \dd Y_t - \tfrac12\operatorname{Tr}( P_t H_t'  H_t) \, \dd t + \tfrac12\operatorname{Tr}(\tilde a_t  P_t^{-1}) \, \dd t\\
& +  (\dd X_t-\dd \nu_t)' P_t^{-1} (\nu_t-X_t)  -(\dd X_t - \dd \nu_t)' P_t^{-1} (\dd X_t - \dd \nu_t)/2\\
& +\tfrac12 ( X_t- \nu_t)' P^{-1}_t(B_t P_t + P_tB_t' ) P^{-1}_t (X_t-\nu_t) \, \dd t  -\tfrac12 ( X_t- \nu_t)'  P^{-1}_t \tilde a_t P^{-1}_t (X_t-\nu_t) \, \dd t \\
& +\tfrac12 \left(( X_t- \nu_t)'   H_t'  H_t  (X_t-\nu_t)\right) \, \dd t \\
= &-  (H_t \nu_t)' \, \dd Y_t + \tfrac12 \|H_t \nu_t\|^2 \, \dd t\\
& + \operatorname{Tr} (H_t P_t H_t') \, \dd t - \tfrac12\operatorname{Tr}( P_t H_t'  H_t) \, \dd t + \tfrac12\operatorname{Tr}(\tilde a_t  P_t^{-1}) \, \dd t\\
& +  (b(t, X_t) \, \dd t   - (X\nu_t + m) \, \dd t )' P_t^{-1} (\nu_t-X_t) \\
& +  (H_t(\nu_t-X_t))' \, \dd Y_t    - ( H_t'   H_t \nu_t )'  (\nu_t-X_t) \, \dd t\\
& +  ( \sigma(t, X_t) \, \dd W_t )' P_t^{-1} (\nu_t-X_t) \\
& -(\dd X_t )' P_t^{-1} (\dd X_t  )/2 -( \dd \nu_t)' P_t^{-1} (\dd \nu_t  )/2\\
& -(X( X_t- \nu_t))'P^{-1}_t (\nu_t - X_t) \, \dd t  -\tfrac12 ( X_t- \nu_t)'  P^{-1}_t  \tilde a_t P^{-1}_t (X_t-\nu_t) \, \dd t \\
& + \tfrac12 X_t '  H_t'  H_t  X_t \, \dd t -  \nu_t'  H_t'  H_t  X_t \, \dd t   + \tfrac12 \nu_t'   H_t'  H_t    \nu_t \, \dd t \\
= &-  (H_t X_t)' \, \dd Y_t + \tfrac12 \|  H_t  X_t \|^2 \, \dd t\\
& +  (b(t, X_t) \, \dd t   - (XX_t + m) \, \dd t )' P_t^{-1} (\nu_t-X_t) \\
& - \tfrac12\operatorname{Tr}((a(t, X_t) - \tilde a) P_t^{-1}) \, \dd t +\tfrac12 ( X_t- \nu_t)'  P^{-1}_t (-a- \tilde a) P^{-1}_t (X_t-\nu_t) \, \dd t \\
& +  (\sigma'(t, X_t) P_t^{-1} (\nu_t-X_t))' \, \dd W_t  +\tfrac12 \|\sigma (s, X_t) P^{-1}_t (X_t-\nu_t)\|^2 \, \dd t \qedhere
\end{align*} 
}
\end{proof}

\begin{proof}[Proof of proposition~\ref{balance}]

To establish detailed balance,
and therefore prove the proposition, by \cite{10.1214/aoap/1027961031}, page 2, we need to show 
\[
\alpha(z, z') \pi(\dd z) Q(z, \dd z')  = \alpha(z', z) \pi(\dd z') Q(z', \dd z). 
\]
In our case,
\begin{align*}
\alpha(z, z') \pi(\dd z) Q(z, \dd z') &= \alpha(z, z') \frac{\dd\pi}{\dd \P_W}(z) \P_W(\dd z)Q(z, \dd z')  \\
&= \alpha(z', z) \frac{\dd\pi}{\dd \P_W}(z') \P_W(\dd z)Q(z, \dd z') \\
&= \alpha(z', z) \frac{\dd\pi}{\dd \P_W}(z') \P_W(\dd z')Q(z', \dd z)\\
&=\alpha(z', z) \pi(\dd z') Q(z', \dd z), 
\end{align*}
follows from $\min(a/b, 1)\cdot b = \min(b/a, 1)\cdot a$ and the symmetry of the measure 
\[
\lambda(\dd w, \dd w') = \P_W(\dd w) Q(w, \dd w').\qedhere
\]

\end{proof}

\end{document}